\documentclass[sn-mathphys-num]{sn-jnl}

\usepackage{graphicx}%
\usepackage{multirow}%
\usepackage{amsmath,amssymb,amsfonts}%
\usepackage{amsthm}%
\usepackage{mathrsfs}%
\usepackage[title]{appendix}%
\usepackage{xcolor}%
\usepackage{textcomp}%
\usepackage{manyfoot}%
\usepackage{booktabs}%
\usepackage{algorithm}%
\usepackage{algorithmicx}%
\usepackage{algpseudocode}%
\usepackage{listings}%

\theoremstyle{thmstyleone}%
\newtheorem{theorem}{Theorem}
\newtheorem{assumption}[theorem]{Assumption}
\newtheorem{lemma}[theorem]{Lemma}

\theoremstyle{thmstyletwo}%
\newtheorem{example}{Example}%
\newtheorem{remark}{Remark}%

\theoremstyle{thmstylethree}%
\newtheorem{definition}{Definition}%

\allowdisplaybreaks[4]

\begin{document}

\title[Strong convergence of the adaptive Milstein method for nonlinear SDEPCAs]{Strong convergence of the adaptive Milstein method for nonlinear stochastic differential equations with piecewise continuous arguments}

\author[1,2]{\fnm{Yuhang} \sur{Zhang}}\email{zhangyh@hit.edu.cn}

\author*[3]{\fnm{Minghui} \sur{Song}}\email{songmh@hit.edu.cn}

\author[4,2]{\fnm{Jiaqi} \sur{Zhu}}\email{zhujq@hit.edu.cn}

\affil[1]{\orgdiv{School of Astronautics}, \orgname{Harbin Institute of Technology}, \orgaddress{\city{Harbin},\postcode{150080}, \country{China}}}

\affil[2]{\orgdiv{Zhengzhou Research Institute}, \orgname{Harbin Institute of Technology}, \orgaddress{\city{Zhengzhou}, \postcode{450000}, \country{China}}}

\affil[3]{\orgdiv{School of Mathematics}, \orgname{Harbin Institute of Technology}, \orgaddress{\city{Harbin},  \postcode{150001}, \country{China}}}

\affil[4]{\orgdiv{National Key Laboratory of Science and Technology on Advanced Composites in Special Environments}, \orgname{Harbin Institute of Technology}, \orgaddress{\city{Harbin}, \postcode{150080}, \country{China}}}


\abstract{In this work, an adaptive time-stepping Milstein method is constructed for stochastic differential equations with piecewise continuous arguments (SDEPCAs), where the drift is one-sided Lipschitz continuous and the diffusion does not impose the commutativity condition. It is widely recognized that explicit Euler or Milstein methods may blow up when the system exhibits superlinear growth, and modifications are needed. Hence we propose an adaptive variant to deal with the case of superlinear growth drift coefficient. To the best of our knowledge, this is the first work to develop a numerical method with variable step sizes for nonlinear SDEPCAs. It is proven that the adaptive Milstein method is strongly convergent in the sense of $\mathcal{L}_p, p\ge 2$, and the convergence rate is optimal, which is consistent with the order of the explicit Milstein scheme with globally Lipschitz coefficients. Finally, several numerical experiments are presented to support the theoretical analysis.}


\maketitle

\section{Introduction}\label{Introduction}
In this paper, we focus on the numerical approximation for the $n$-dim stochastic differential equations with piecewise continuous arguments (SDEPCAs), which are given by
\begin{align*}
{\rm d}x(t)=f(x(t),x([t])){\rm d}t+ g(x(t),x([t])){\rm d}B(t), ~\forall t\in[0,T]
\end{align*}
where $T>0$ and the initial value is $x(0)=x_0\in\mathbb{R}^n$, with $[\cdot]$ denoting the greatest-integer function. The study of differential equations with piecewise continuous arguments was initiated in
\cite{Shah1983,Wiener1983,Cooke1984,Wiener1984}, and the general theory and basic results have been systematically introduced in Wiener's book \cite{Wienerbook}. SDEPCAs extend these models by incorporating the effect of noise and have been widely applied in various fields, including neural networks (\cite{XiaoaiLi2014}) and control theory, especially in the study of stabilization for stochastic differential equations (SDEs) via feedback control based on discrete-time state observations (see, e.g., \cite{X.R.Mao2013,LIU2020,XiaoyueLi2020SIAM,WuHao2020} and references therein). To be specific, the general form of stochastic controlled systems is
\begin{align*}
{\rm d}x(t)=f(t,x(t),u(t)){\rm d}t+g(t,x(t),u(t)){\rm d}B(t),
\end{align*}
where $u$ is the control input, while the control problem of stochastic systems by discrete-time state feedback need to establish the stability of 
  \begin{align*}
{\rm d}x(t)=f(t,x(t),u(x([t/\tau]\tau)){\rm d}t+g(t,x(t),u(x([t/\tau]\tau))){\rm d}B(t),
\end{align*}
which is a typical SDEPCA. However, most such equations do not have explicit solutions, making numerical methods essential for understanding their dynamics. 

Some recent studies have explored numerical approximations for SDEPCAs, with most focusing on Euler and Euler-type methods using uniform step sizes, such as the split-step theta method (\cite{Lu2017}), the tamed Euler method (\cite{yuhZhang2022}), and the truncated Euler method (\cite{geng2021}), among others (see, e.g. \cite{Marija2016,XIE2019,XIE20201}). The convergence orders of these methods typically do not exceed one-half. However, in order to achieve the required accuracy in many real-world problems, higher-order numerical methods are also an important part of the numerical analysis. To the best of our knowledge, \cite{yuhZhang2023numerical_algorithms} is the only work that has constructed an explicit Milstein method for SDEPCAs, proving a convergence rate of 1 under global Lipschitz condition. However, \cite{Martin2011} pointed out that Euler and Milstein methods with uniform step sizes cannot converge for a large class of SDEs with superlinearly growing coefficients. Therefore, our goal in this work is to propose an adaptive variant of the explicit Milstein method that achieves convergence of order 1 in the sense of $\mathcal{L}_p$ to nonlinear SDEPCAs. 

Adaptive time-stepping methods have gained attention in recent years for solving nonlinear SDEs (for example, \cite{Lamba2007,Lord2018,FangWei2020,Reisinger2022} and so on), the idea is to avoid the oscillatory superlinearly growth of numerical solutions by dynamically adjusting the step size rather than modifying the coefficients directly. 

Next, we outline our contributions in comparison to existing literature.

\begin{itemize}
\item We propose an explicit Milstein method for SDEPCAs, where the drift can grow superlinearly and the diffusion does not  satisfy the commutativity condition. We prove that the method achieves a strong convergence rate of 1. To the best of our knowledge, \cite{yuhZhang2023numerical_algorithms} is the only work that has constructed a first-order method for SDEPCAs, but it requires the coefficients to satisfy global Lipschitz conditions and assumes commutative diffusion. For SDEPCAs with non-global Lipschitz coefficients, all existing numerical methods converge with an order of $1/2$. 
\item To the best of our knowledge, this is the first work to develop an adaptive time-stepping method for SDEPCAs. In contrast, all numerical methods for SDEPCAs in existing literature use a uniform step size (see, e.g., \cite{XIE20201,yuhZhang2022,MR4489747} and references therein). In fact, most current research focuses on adaptive methods for SDEs without delay (see \cite{FangWei2020,Reisinger2022,Conall2019arxiv} and references there in). There is limited work on adaptive method for SDEs with delay, to our knowledge, \cite{MR3339209} is the only one, but it focuses on weak convergence for SDEs with constant delay under global Lipschitz conditions. The SDEPCAs we study here are a class of SDEs with variable delay, hence, this work can further enriches the field of adaptive methods. 

\item An adaptive time-stepping Milstein method has been shown in \cite{Conall2019arxiv} to be strongly convergent with an order of 1 for SDEs without delay, but it relies on a backstop scheme, i.e., at time $t_n$ , a single-step backstop method is used if the step size $h_n$ is outside the specified range. In this work, we develop the adaptive variant following the ideas of \cite{FangWei2020}, without using the backstop scheme. Additionally, the convergence rate in \cite{Conall2019arxiv} is proven in the sense of $\mathcal{L}_2$, while we prove strong convergence with an order of 1 in the sense of $\mathcal{L}_p$ for $p\ge 2$.
\end{itemize}

The remainder of this article is organized as follows: In Section 2, we introduce notations and settings. In Section 3, we present our adaptive Milstein scheme and prove its boundedness, with the time-stepping strategy inspired by \cite{FangWei2020}. The main convergence result is given in Section 4. Finally, in Section 5, we present some numerical experiments to validate the theoretical results.

\section{Notations and preliminaries}\label{Notations and preliminaries}

Throughout this paper, unless otherwise specified, we use the following notations. For two real numbers $a$ and $b$, we denote $a\vee b$ and $a\wedge b$ as $\max\left\{a,b\right\}$ and $\min\left\{a,b\right\}$, respectively. $\mathbb{N}$ and $\mathbb{N}_+$ represent the sets $\{0,1,2,\dots\}$ and $\{1,2,3,\dots\}$, respectively. $\vert x\vert$ denotes the Euclidean vector norm, and $\langle x,y\rangle$ represents the inner product of vectors $x$ and $y$. If $A$ is a vector or matrix, its transpose is denoted by $A^{\rm T}$. For a matrix $A\in\mathbb{R}^{n\times d}$, its Frobenius norm is defined as $\|A\|\triangleq\sqrt{\sum_{i=1}^n\sum_{j=1}^dA_{ij}^2}$. We also use $\|\cdot\|_{T_3}$ to denote the tensor norm of a rank-3 tensor in $\mathbb{R}^{n\times d\times m}$. 

For all $x,y\in\mathbb{R}^n$ and any $\phi(x,y)\in C^{2,0}(\mathbb{R}^n\times\mathbb{R}^n,\mathbb{R}^n)$, we denote $D_x\phi(x,y)$ as the matrix $\left(\frac{\partial \phi_i(x,y)}{\partial x_j}\right)_{n\times n}$. The second derivative of $\phi(x,y)$ with respect to $x$ forms a rank-3 tensor, denoted by $D_{xx}\phi(x,y)$.

Moreover, let $(\Omega ,\mathcal{F},\left\{\mathcal{F}_t\right\}_{t\ge 0},\mathbb{P})$ be a complete probability space with a filtration $\left\{\mathcal{F}_t\right\}_{t\ge 0}$ satisfying the usual conditions (i.e., it is right continuous and $\mathcal{F}_0$ contains all $\mathbb{P}$-null sets), and let $\mathbb{E}$ denote the expectation corresponding to $\mathbb{P}$.  Let $B(t)=(B^1(t),\dots,B^d(t))^{\rm T}$ be a d-dimensional Brownian motion. For any $T>0$, we consider the following SDEPCA
\begin{align}\label{SDEPCA}
{\rm d}x(t)=f(x(t),x([t])){\rm d}t+ \sum_{j=1}^d g_j(x(t),x([t])){\rm d}B^j(t)
\end{align}
on $t\in [0,T]$ with initial data $x(0)=x_0\in\mathbb{R}^n$, where $x(t)\in\mathbb{R}^n$, $f:\mathbb{R}^n\times \mathbb{R}^n\to \mathbb{R}^n , g_j:\mathbb{R}^n\times \mathbb{R}^n\to \mathbb{R}^{n}$, $j=1,2,\dots,d$, 



The definition of the solution of SDEPCAs can be found in \cite{songmh2012}.  In this work, we assume that the coefficients of \eqref{SDEPCA} satisfy the following conditions.
\begin{assumption}\label{Local Lipschitz condition}
For every positive number $R$, there exists a positive constant $L_R$ such that 
\begin{align*}
\vert f(x, y)-f(\bar{x},\bar{y})\vert \vee\vert g_j(x,y)-g_j(\bar{x},\bar{y})\vert  \le L_R(\vert x-\bar{x}\vert  +\vert y-\bar{y}\vert)
\end{align*} 
for those $x, y,\bar{x},\bar{y}\in\mathbb{R}^n$ with $\vert x\vert\vee\vert y\vert\vee \vert  \bar{x}\vert\vee \vert \bar{y}\vert\le R$ and all $j=1,\dots, d$.
\end{assumption}
\begin{assumption}\label{f condition}
There exists a constant $L>0$ such that 
\begin{align*}
\langle x,f(x, y)\rangle \le& L(1+|x|^2+|y|^2)
\end{align*} 
for all $x, y\in\mathbb{R}^n$.
\end{assumption}

\begin{assumption}\label{g condition}
There exist constant $K_1,K_2>0$ such that 
\begin{align*}
&\vert g_j(x, y)\vert \le K_1(1+|x|+|y|),\\
&\|D_{x} g_j(x, y)\| \le K_2
\end{align*} 
for all $x, y\in\mathbb{R}^n$ and $j=1,2,\dots,d$.
\end{assumption}

Under Assumption \ref{g condition}, with the help of the Cauchy-Schwarz inequality, it is easy to see that 
\begin{align}\label{D_xg_jg_r}
\vert D_{x} g_j(x, y)g_r(x,y)\vert \le \|D_{x} g_j(x, y)\|\vert g_r(x, y)\vert \le K(1+|x|+|y|)
\end{align} 
for any $j,r=1,2,\dots,d$ with $K=K_1K_2$.

According to the Theorem 3.1 in \cite{songmh2012}, we can obtain the following existence and uniqueness theorem.
\begin{lemma}\label{exact solu. bounded}
Under Assumptions \ref{Local Lipschitz condition}-\ref{g condition}, there is a unique global solution $x(t)$ to Eq.\eqref{SDEPCA} on $t\ge 0$ with initial data $x(0)=x_0$. Moreover, for any $p>0$, the solution has the property that
\begin{align*}
 \mathbb{E}|x(t)|^p<\infty, ~\forall t\ge 0.
 \end{align*}
\end{lemma}

\section{The Milstein scheme}\label{Milstein scheme}

We will define the adaptive Milstein scheme for Eq.\eqref{SDEPCA} and consider its boundedness in this section. Adaptive time-stepping methods usually lead to a sequence of time steps $\{\Delta_k\}_{k\in\mathbb{N}}$, which yield a partition of the time interval $0=t_0<t_1<t_2<\cdots< T$, where we set $t_{k+1}=t_0+\sum_{i=0}^{k}\Delta_i, k\in\mathbb{N}$ with $t_0=0$. First, let us introduce the definition of $\{\Delta_k\}_{k\in\mathbb{N}}$: Write $X_k$ as the numerical approximation of $x(t)$ at $t=t_k, k\in\mathbb{N}$, we use $\Delta(X_k)$ to ensure the boundedness of the Milstein method when $\vert X_k\vert$ is large, $\Delta(\cdot)$ has to satisfy the following condition:

\begin{assumption}\label{step size}
Suppose that $\Delta: \mathbb{R}^n\to (0,T]$ denotes the adaptive time step function, which is continuous and strictly positive, and there exists a positive constant $\alpha$ such that for all $x,y\in\mathbb{R}^n$,
\begin{align}\label{step size equ.}
\langle x,f(x,y)\rangle+\frac{1}{2}\Delta(x)\vert f(x,y)\vert^2\le \alpha(1+|x|^2+|y|^2).
\end{align}
\end{assumption}
\begin{remark}
Under Assumption \ref{step size}, \eqref{step size equ.} also holds for any other step size $\Delta\le \Delta(x)$, i.e.,
\begin{align*}
\langle x,f(x,y)\rangle+\frac{1}{2}\Delta\vert f(x,y)\vert^2\le \alpha(1+|x|^2+|y|^2),~\forall \Delta\le \Delta(x).
\end{align*}
\end{remark}

Standard convergence analysis for numerical methods with a uniform time stepsize $\Delta$ considers the limit $\Delta\to0$, it obviously needs to be modified when using an adaptive time step. To this end, for any given $M\in\mathbb{N}_+$, define
\begin{align}\label{Delta_k}
\Delta_k=\min\left\{\frac{\Delta(X_k)}{M},[t_k]+1-t_k\right\},
\end{align}
then we will consider the limit $M\to \infty$ in the convergence analysis. 

 \begin{remark}
According to the definition of $\Delta(\cdot)$ and $\Delta_k$, it is easy to know that 
\begin{align*}
\Delta_k\le \frac{\Delta(X_k)}{M}\le \frac{T}{M}\quad{\rm and}\quad0<\Delta_k\le 1.
\end{align*}
 \end{remark}
 
 \begin{remark}
For any $t\in[0,T]$, $[t]$ is a nonnegative integer in Eq.\eqref{SDEPCA}, and the term $[t_k]+1-t_k$ in \eqref{Delta_k} is to ensure all positive integers can be the discretization times.
 \end{remark}
 
In order to ensure that the simulation on the interval $[0,T]$ can be completed in a finite number of time steps, we require an assumption on the lower bound of stepsizes:
\begin{assumption}\label{step size lower bound}
There exist positive constants $a,b$ and $q$ such that the adaptive time step function satisfies
\begin{align*}
\Delta(x)\ge (a|x|^q+b)^{-1}, ~\forall x\in\mathbb{R}^n.
\end{align*}
\end{assumption}

\begin{definition}[\cite{Lord2018,Maobook}]
Suppose that each member of the sequence $\{t_k\}_{k\in\mathbb{N}}$ is an $\mathcal{F}_t$-stopping time: i.e., $\{t_k \le t\}\in\mathcal{F}_t$ for all $t\ge 0$, where $\{\mathcal{F}_t\}_{t\ge0}$ is the natural filtration. We may then define a discrete-time filtration $\{\mathcal{F}_{t_k}\}_{k\in\mathbb{N}}$ by
\begin{align*}
\mathcal{F}_{t_k}=\{A\in \mathcal{F}: A\cap\{t_k \le t\}\in\mathcal{F}_{t} \text{~for all~} t\ge 0\},~k\in\mathbb{N}.
\end{align*}
\end{definition}


Next we define the adaptive Milstein method used in this paper for SDEPCA \eqref{SDEPCA}. For $x,y\in\mathbb{R}^n, j,r=1,2,\dots, d$, and $k\in\mathbb{N}$, define 
\begin{align*}
m_k=&\max\{m\in\mathbb{N}:t_m\le t_k,t_m\in\mathbb{N}\},\\
I_{rj}^{t_k,t}=&\int_{t_{k}}^{t}\int_{t_{k}}^u {\rm d}B^r(v){\rm d}B^j(u).
\end{align*}
For any $k\in\mathbb{N}$, the adaptive Milstein solution for Eq.\eqref{SDEPCA} is given by 
\begin{align}\label{discrete Milstein-1}
X_{k+1}=&X_k+ f\left(X_k,X_{m_k}\right)\Delta_k+ \sum_{j=1}^d g_j\left(X_k,X_{m_k}\right)\Delta B^j(t_{k+1})\notag\\
&+\sum_{j,r=1}^d D_xg_j\left(X_k,X_{m_k}\right)g_r\left(X_k,X_{m_k}\right)I_{rj}^{t_k,t_{k+1}}
\end{align}
with $X(0)=x(0)=x_0, \Delta B^j(t_{k+1})=B^j(t_{k+1})-B^j(t_k)$.   For $t\in[t_k,t_{k+1})$, define 
\begin{align}\label{continuous Milstein}
X(t)=&X_k+\int_{t_k}^t f(X_k,X_{m_k}) {\rm d}u+ \sum_{j=1}^d\int_{t_k}^t g_j(X_k,X_{m_k}) {\rm d}B^j(u)\notag\\
&+\sum_{j,r=1}^d D_xg_j(X_k,X_{m_k})g_r(X_k,X_{m_k})I_{rj}^{t_k,t}.
\end{align}
Let $\bar{X}(t)=X_k$, $\underline{t}=t_k, \forall t\in[t_k,t_{k+1})$, it follows
\begin{align*}
X(t)=&X_0+\int_{0}^t f(\bar{X}(u),\bar{X}([u])) {\rm d}u+ \sum_{j=1}^d\int_{0}^t g_j(\bar{X}(u),\bar{X}([u])) {\rm d}B^j(u)\notag\\
&+\sum_{j,r=1}^d \int_{0}^t D_xg_j(\bar{X}(u),\bar{X}([u]))g_r(\bar{X}(u),\bar{X}([u]))\Delta B^r(u){\rm d}B^j(u), \forall t\ge0
\end{align*}
from equation \eqref{continuous Milstein}, where $\Delta B^r(u)=B^r(u)-B^r(\underline{u})$. It is easy to see that $X(t_k)=\bar{X}(t_k)=X_k$.  For $t\in[t_k,t_{k+1}]$ and $r\neq j$, due to 
\begin{align*}
I_{rj}^{t_k,t}+I_{jr}^{t_k,t}=(B^j(t)-B^j(t_k))(B^r(t)-B^r(t_k))=\Delta B^j(t)\Delta B^r(t),
\end{align*}
the last terms in \eqref{discrete Milstein-1} and \eqref{continuous Milstein} can both be divided into
\begin{align}\label{I_rj^t_k,t}
&\sum_{j,r=1}^d D_xg_j(X_k,X_{m_k})g_r(X_k,X_{m_k}) I_{rj}^{t_k,t}\notag\\
&=\frac{1}{2}\sum_{j=1}^d D_xg_j(X_k,X_{m_k})g_j(X_k,X_{m_k})\Big((\Delta B^j(t))^2-(t-t_k)\Big)\notag\\
&\quad+\frac{1}{2}\sum_{\substack{j,r=1\\ j<r}}^d \Big(D_xg_j(X_k,X_{m_k})g_r(X_k,X_{m_k})+D_xg_r(X_k,X_{m_k})g_j(X_k,X_{m_k})\Big)\Delta B^j(t)\Delta B^r(t)\notag\\
&\quad+\sum_{\substack{j,r=1\\ j<r}}^d \Big(D_xg_r(X_k,X_{m_k})g_j(X_k,X_{m_k})-D_xg_j(X_k,X_{m_k})g_r(X_k,X_{m_k})\Big)A_{jr}^{t_k,t},
\end{align}
where $A_{jr}^{t_k,t}=\frac{1}{2}(I_{jr}^{t_k,t}-I_{rj}^{t_k,t})$ is the L\'evy area (\cite{Conall2019arxiv,Shiryaev1996}).
%
%

According to Lemma 2.2 and  Remark 3.2 in \cite{Conall2019arxiv}, we have the following results.

\begin{lemma}[\cite{Conall2019arxiv}]\label{Levy area approximation}
For all $j,r=1,2,\dots,d$ and $0\le t_k\le t\le T$, there exists a positive constant $C$ such that for $p\ge 1$, it holds that
\begin{align*}
\mathbb{E}\left(\vert A_{jr}^{t_k,t}\vert^p\big\vert\mathcal{F}_{t_k}\right)\le C(t-t_k)^{p},~a.s.
\end{align*}
\end{lemma}

\begin{lemma}[\cite{Conall2019arxiv,Shiryaev1996}]\label{|B(t_1)-B(t_2)|^p}
Let $t_k,t_{k+1}$ be two $\mathcal{F}_{t_k}$ stopping times and $\mathcal{F}_{t_k}$-measurable. Then for any $t\in[t_k,t_{k+1}]$, $\Delta B^j(t)= B^j(t)-B^j(t_k), j=1,2,\dots,d$ is $\mathcal{F}_{t_k}$-conditionally normally distributed and 
\begin{align*}
\mathbb{E}\left( \Delta B^j(t) \bigg\vert\mathcal{F}_{t_k}\right)=&0,~a.s.,\\
\mathbb{E}\left(\left\vert \Delta B^j(t)\right\vert^p\bigg\vert\mathcal{F}_{t_k}\right)= &\gamma_p( t-t_k)^{\frac{p}{2}}, ~a.s.
\end{align*}
with $\gamma_p=2^{\frac{p}{2}}\Gamma(\frac{p+1}{2})/\sqrt{\pi}$ for $p=1,2,\dots$, $\Gamma$ is the Gamma function.
\end{lemma}
In the rest of this paper, we will always use $C$ to denote a generic constant that varies from one place to another and may depend on $p$ and $T$, but independent of the step sizes.

\begin{theorem}\label{numerical solu. bounded}
Suppose that Eq.\eqref{SDEPCA} satisfies Assumptions \ref{f condition} and \ref{g condition}, $\Delta_k$ is taken as form \eqref{Delta_k} with $\Delta(\cdot)$ satisfies Assumption \ref{step size}. Then for any $t\in[t_k,t_{k+1})$ and $p>0$, the Milstein scheme has the following properties
\begin{align}\label{E(X(t)|F_(t_k))}
 \mathbb{E}\left(\vert X(t)\vert^p\vert\mathcal{F}_{t_k}\right)\le C(1+\vert X_k\vert^p+\vert X([t_k])\vert^p),~a.s.
\end{align}
and
\begin{align*}
\sup_{0\le t\le T}\mathbb{E}\vert X(t)\vert^p\le C.
\end{align*}
\end{theorem}
\begin{proof}\rm
In the following, we suppose that $p\ge 4$, the result for $0 < p < 4$ follows from H\"older's inequality immediately. 

For $x,y\in\mathbb{R}^n, h>0, W_j, Z_{rj}\in\mathbb{R}$, $j,r=1,\dots, d$, let $\phi(x,y,h)=x+ f\left(x,y\right)h$ and
\begin{align*}
\theta(x,y,W_j,Z_{rj})=\sum_{j=1}^d g_j\left(x,y\right)W_j+\sum_{j,r=1}^d D_xg_j(x,y)g_r(x,y) Z_{rj},
\end{align*}
 then for any $t\in [t_k,t_{k+1}], k\in\mathbb{N}$, according to \eqref{discrete Milstein-1} and \eqref{continuous Milstein}, we have
\begin{align}\label{X(t)^2_0}
\vert X(t)\vert^2=&\vert \phi(X_k,X_{m_k},t-t_k)\vert^2+\left\vert \theta(X_k,X_{m_k},\Delta B^j(t),I_{rj}^{t_k,t})\right\vert^2\notag\\
&+2\left\langle \phi(X_k,X_{m_k},t-t_k), \theta(X_k,X_{m_k},\Delta B^j(t),I_{rj}^{t_k,t})\right\rangle.
\end{align}
Since $t-t_k\le \Delta_k\le \Delta(X_k)$, it follows from Assumption \ref{step size} that
\begin{align}\label{phi^2}
\vert\phi(X_k,X_{m_k},t-t_k)\vert^2\le &\vert X_k\vert^2+2(t-t_k)\left(\langle X_k, f(X_k,X_{m_k})\rangle+\frac{1}{2}\vert f(X_k,X_{m_k})\vert^2\Delta(X_k)\right)\notag\\
\le &\vert X_k\vert^2+2\alpha(t-t_k)(1+\vert X_k\vert^2+\vert X_{m_k}\vert^2).
\end{align}
Following \eqref{I_rj^t_k,t}, using the inequality $|\sum_{i=1}^na_i|^p\le n^{p-1}\sum_{i=1}^n |a_i|^p, n\in\mathbb{N}_+$, we can obtain 
\begin{align*}
&\left\vert \theta(X_k,X_{m_k},\Delta B^j(t),I_{rj}^{t_k,t})\right\vert^2\\
\le&5\left\vert \sum_{j=1}^dg_j(X_k,X_{m_k})\Delta B^j(t)\right\vert^2+\frac{5}{4}\bigg\vert\sum_{j=1}^d D_xg_j(X_k,X_{m_k})g_j(X_k,X_{m_k})\bigg\vert^2(t-t_k)^2\\
&+\frac{5}{4}\bigg\vert\sum_{j=1}^d D_xg_j(X_k,X_{m_k})g_j(X_k,X_{m_k})(\Delta B^j(t))^2\bigg\vert^2\\
&+\frac{5}{4}\Bigg\vert\sum_{\substack{j,r=1\\ j<r}}^d \big(D_xg_j(X_k,X_{m_k})g_r(X_k,X_{m_k})+D_xg_r(X_k,X_{m_k})g_j(X_k,X_{m_k})\big)\Delta B^j(t)\Delta B^r(t)\Bigg\vert^2\\
&+5\Bigg\vert\sum_{\substack{j,r=1\\ j<r}}^d\left(D_xg_r(X_k,X_{m_k})g_j(X_k,X_{m_k})-D_xg_j(X_k,X_{m_k})g_r(X_k,X_{m_k})\right)A_{jr}^{t_k,t}\Bigg\vert^2.
\end{align*}
Then applying the Cauchy-Schwarz inequality and Assumption \ref{g condition}, it can be deduced that
\begin{align}\label{X(t)^2_2}
&\left\vert \theta(X_k,X_{m_k},\Delta B^j(t),I_{rj}^{t_k,t})\right\vert^2\notag\\
\le&5d\sum_{j=1}^d\vert g_j(X_k,X_{m_k})\vert^2\vert \Delta B^j(t)\vert^2+\frac{5d}{4}\sum_{j=1}^d\vert D_xg_j(X_k,X_{m_k}) g_j(X_k,X_{m_k})\vert^2\vert \Delta B^j(t)\vert^4\notag\\
&+\frac{5d^2}{4}\sum_{\substack{j,r=1\\ j<r}}^d\vert D_xg_j(X_k,X_{m_k})g_r(X_k,X_{m_k})+D_xg_r(X_k,X_{m_k})g_j(X_k,X_{m_k})\vert^2\vert \Delta B^j(t)\vert^2\vert \Delta B^r(t)\vert^2\notag\\
&+\frac{5d}{4}\sum_{j=1}^d\vert D_xg_j(X_k,X_{m_k}) g_j(X_k,X_{m_k})\vert^2(t-t_k)^2\notag\\
&+5d^2\sum_{\substack{j,r=1\\ j<r}}^d\vert D_xg_r(X_k,X_{m_k})g_j(X_k,X_{m_k})-D_xg_j(X_k,X_{m_k})g_r(X_k,X_{m_k})\vert^2\vert A_{jr}^{t_k,t}\vert^2\notag\\
\le&C(1+\vert X_k\vert^2+\vert X_{m_k}\vert^2)\bigg(\sum_{j=1}^d\vert \Delta B^j(t)\vert^2+(t-t_k)^2+\sum_{j=1}^d\vert \Delta B^j(t)\vert^4+\sum_{\substack{j,r=1\\ j<r}}^d\vert A_{jr}^{t_k,t}\vert^2\bigg).
\end{align}
Using \eqref{I_rj^t_k,t} once again, we know that
\begin{align*}
&2\left\langle \phi(X_k,X_{m_k},t-t_k), \theta(X_k,X_{m_k},\Delta B^j(t),\int_{t_{k}}^{t}\int_{t_{k}}^u {\rm d}B^r(v){\rm d}B^j(u))\right\rangle\\
=&2\sum_{j=1}^d \int_{t_k}^t\phi^{\rm T}(X_k,X_{m_k},t-t_k)g_j(X_k,X_{m_k}){\rm d}B^j(u)\\
&+\sum_{j=1}^d\phi^{\rm T}(X_k,X_{m_k},t-t_k) D_xg_j(X_k,X_{m_k})g_j(X_k,X_{m_k})(\Delta B^j(t))^2\\
&-\sum_{j=1}^d\int_{t_k}^t\phi^{\rm T}(X_k,X_{m_k},t-t_k) D_xg_j(X_k,X_{m_k})g_j(X_k,X_{m_k}){\rm d}u\\
&+\sum_{\substack{j,r=1\\ j<r}}^d\phi^{\rm T}(X_k,X_{m_k},t-t_k) \Big(D_xg_j(X_k,X_{m_k})g_r(X_k,X_{m_k})+D_xg_r(X_k,X_{m_k})g_j(X_k,X_{m_k})\Big)\Delta B^r(t)\Delta B^j(t)\\
&+2\sum_{\substack{j,r=1\\ j<r}}^d\phi^{\rm T}(X_k,X_{m_k},t-t_k) \left(D_xg_r(X_k,X_{m_k})g_j(X_k,X_{m_k})-D_xg_j(X_k,X_{m_k})g_r(X_k,X_{m_k})\right)A_{rj}^{t_k,t},
\end{align*}
by the fundamental inequality $2ab\le \vert a\vert^2+\vert b\vert^2$, the Cauchy-Schwarz inequality, \eqref{phi^2}, Assumption \ref{g condition} and  \eqref{D_xg_jg_r}, we have
\begin{align}\label{X(t)^2_3}
&2\left\langle \phi(X_k,X_{m_k},t-t_k), \theta(X_k,X_{m_k},\Delta B^j(t),\int_{t_{k}}^{t}\int_{t_{k}}^u {\rm d}B^r(v){\rm d}B^j(u))\right\rangle\notag\\
\le& \sum_{j=1}^d\int_{t_k}^t \left(\vert \phi(X_k,X_{m_k},t-t_k)\vert^2+\vert g_j(X_k,X_{m_k})\vert^2\right){\rm d}B^j(u)\notag\\
&+\frac{1}{2}\sum_{j=1}^d\left(\vert\phi(X_k,X_{m_k},t-t_k)\Delta B^j(t) \vert^2+\vert D_xg_j(X_k,X_{m_k})g_j(X_k,X_{m_k})\Delta B^j(t)\vert^2\right)\notag\\
&+\frac{1}{2}\sum_{j=1}^d \int_{t_k}^t\left(\vert \phi(X_k,X_{m_k},t-t_k)\vert^2+\vert D_xg_j(X_k,X_{m_k})g_j(X_k,X_{m_k})\vert^2\right){\rm d}u\notag\\
&+\frac{d}{2}\sum_{j=1}^d\left\vert \phi(X_k,X_{m_k},t-t_k) \Delta B^j(t)\right\vert^2\notag\\
&+\frac{1}{2}\sum_{j,r=1}^d\left\vert  (D_xg_j(X_k,X_{m_k})g_r(X_k,X_{m_k})+D_xg_r(X_k,X_{m_k})g_j(X_k,X_{m_k}))\Delta B^j(t)\right\vert^2\notag\\
&+\sum_{\substack{j,r=1\\ j<r}}^d\left(\vert \phi^{\rm T}(X_k,X_{m_k},t-t_k)\vert^2+ \vert D_xg_r(X_k,X_{m_k})g_j(X_k,X_{m_k})-D_xg_j(X_k,X_{m_k})g_r(X_k,X_{m_k})\vert^2\right)\vert A_{rj}^{t_k,t}\vert\notag\\
\le& C (1+\vert X_k\vert^2+\vert X_{m_k}\vert^2)\left(\sum_{j=1}^d\Delta B^j(t)+\sum_{j=1}^d\left\vert \Delta B^j(t)\right\vert^2+(t-t_k)+\sum_{\substack{j,r=1\\ j<r}}^d\vert A_{rj}^{t_k,t}\vert\right).
\end{align}
Substituting \eqref{phi^2}-\eqref{X(t)^2_3} into \eqref{X(t)^2_0}, note that $t-t_k\le \Delta_k\le 1$, one can obtain 
\begin{align}\label{X(t)=X_k+}
\vert X(t)\vert^2\le&\vert X_k\vert^2+C (1+\vert X_k\vert^2+\vert X_{m_k}\vert^2)\Bigg((t-t_k)+\sum_{j=1}^d\left\vert \Delta B^j(t)\right\vert^2\notag\\
&+\sum_{j=1}^d\left\vert \Delta B^j(t)\right\vert^4+\sum_{j=1}^d\Delta B^j(t)+\sum_{\substack{j,r=1\\ j<r}}^d\vert A_{rj}^{t_k,t}\vert+\sum_{\substack{j,r=1\\ j<r}}^d\vert A_{rj}^{t_k,t}\vert^2\Bigg).
\end{align}
Hence, using the inequality $(\sum_{i=1}^n\vert a_i\vert)^p\le n^{p-1}\sum_{i=1}^n\vert a_i\vert^p, p>0$ again, it follows that
\begin{align*}
\vert X(t)\vert^p\le&C\vert X_k\vert^p+C (1+\vert X_k\vert^p+\vert X_{m_k}\vert^p)\Bigg((t-t_k)^{\frac{p}{2}}+\sum_{j=1}^d\left\vert \Delta B^j(t)\right\vert^p\\
&+\sum_{j=1}^d\left\vert \Delta B^j(t)\right\vert^{2p}+ \sum_{j=1}^d\vert \Delta B^j(t)\vert^{\frac{p}{2}}+\sum_{\substack{j,r=1\\ j<r}}^d\vert A_{rj}^{t_k,t}\vert^{\frac{p}{2}}+\sum_{\substack{j,r=1\\ j<r}}^d\vert A_{rj}^{t_k,t}\vert^p\Bigg).
\end{align*}
Note that $X_k$, $X_{m_k}$, $t_k$ and $t_{k+1}$ are all $\mathcal{F}_{t_k}$-measurable, taking expectations on both sides, using Lemmas \ref{|B(t_1)-B(t_2)|^p} and \ref{Levy area approximation}, it is easy to see that
\begin{align*}
&\mathbb{E}\left(\vert X(t)\vert^p\vert\mathcal{F}_{t_k}\right)\\
\le&C\vert X_k\vert^p+C (1+\vert X_k\vert^p+\vert X_{m_k}\vert^p)\Bigg(\Delta_k^{\frac{p}{2}}+\sum_{j=1}^d\mathbb{E}\left(\vert \Delta B^j(t)\vert^p\vert\mathcal{F}_{t_k}\right)+\sum_{j=1}^d\mathbb{E}\left(\vert \Delta B^j(t)\vert^{2p}\big\vert\mathcal{F}_{t_k}\right)\\
&+ \sum_{j=1}^d\mathbb{E}\left(\vert \Delta B^j(t)\vert^{\frac{p}{2}}\big\vert\mathcal{F}_{t_k}\right)
+\sum_{\substack{j,r=1\\ j<r}}^d\mathbb{E}\left(\vert A_{rj}^{t_k,t}\vert^{\frac{p}{2}}\big\vert\mathcal{F}_{t_k}\right)+\sum_{\substack{j,r=1\\ j<r}}^d\mathbb{E}\left(\vert A_{rj}^{t_k,t}\vert^p\big\vert\mathcal{F}_{t_k}\right)\Bigg)\\
\le&C\vert X_k\vert^p+C (1+\vert X_k\vert^p+\vert X_{m_k}\vert^p)\Delta_k^{\frac{p}{4}}\\
\le& C (1+\vert X_k\vert^p+\vert X_{m_k}\vert^p),~a.s.
\end{align*}
The proof for \eqref{E(X(t)|F_(t_k))} is completed.

In particular, according to \eqref{X(t)=X_k+}, we also know that
\begin{align}\label{X_k+1^2--1}
\vert X_{k+1}\vert^2\le&\vert X_k\vert^2+C (1+\vert X_k\vert^2+\vert X_{m_k}\vert^2)\Bigg((t_{k+1}-t_k)+\sum_{j=1}^d \vert \Delta B^j(t_{k+1})\vert^2\notag\\
&+\sum_{j=1}^d \vert \Delta B^j(t_{k+1})\vert^4+\sum_{j=1}^d \Delta B^j(t_{k+1}) +\sum_{\substack{j,r=1\\ j<r}}^d\vert A_{rj}^{t_k,t_{k+1}}\vert+\sum_{\substack{j,r=1\\ j<r}}^d\vert A_{rj}^{t_k,t_{k+1}}\vert^2\Bigg).
\end{align}
Since for any $t\in[0,T]$, there always exists $k\in\mathbb{N}$ such that $t\in[t_k,t_{k+1})$, according to \eqref{X(t)=X_k+} and \eqref{X_k+1^2--1}, it can be derived by iteration that 
\begin{align}\label{E|X(t)|^p}
\mathbb{E}\vert X(t)\vert^p=\mathbb{E}(\vert X(t)\vert^2)^{\frac{p}{2}}\le7^{\frac{p}{2}-1} \left(\vert X_0\vert^p+I_1+I_2+I_3+I_4+I_5+I_6\right),
\end{align}
where 
\begin{align*}
I_1=&C\mathbb{E}\Bigg(\sum_{i=0}^{k-1}(1+\vert X_{i}\vert^2+\vert X_{m_i}\vert^2)\Delta_{i}+(1+\vert X_k\vert^2+\vert X_{m_k}\vert^2)(t-t_k)\Bigg)^{\frac{p}{2}},\\
I_2=&C\mathbb{E}\Bigg(\sum_{i=0}^{k-1}(1+\vert X_{i}\vert^2+\vert X_{m_i}\vert^2)\sum_{j=1}^d\vert \Delta B^j(t_{i+1})\vert^2+(1+\vert X_k\vert^2+\vert X_{m_k}\vert^2)\sum_{j=1}^d\vert B^j(t)-B^j(t_k)\vert^2\Bigg)^{\frac{p}{2}},\\
I_3=&C\mathbb{E}\Bigg(\sum_{i=0}^{k-1}(1+\vert X_{i}\vert^2+\vert X_{m_i}\vert^2)\sum_{j=1}^d\vert \Delta B(t_{i+1})\vert^4+(1+\vert X_k\vert^2+\vert X_{m_k}\vert^2)\sum_{j=1}^d\vert B^j(t)-B^j(t_k)\vert^4\Bigg)^{\frac{p}{2}},\\
I_4=&C\mathbb{E}\Bigg(\sum_{i=0}^{k-1}(1+\vert X_{i}\vert^2+\vert X_{m_i}\vert^2)\sum_{j=1}^d\Delta B^j(t_{i+1})+(1+\vert X_k\vert^2+\vert X_{m_k}\vert^2)\sum_{j=1}^d(B^j(t)-B^j(t_k)\Bigg)^{\frac{p}{2}},\\
I_5=&C\mathbb{E}\Bigg(\sum_{i=0}^{k-1}(1+\vert X_{i}\vert^2+\vert X_{m_i}\vert^2)\sum_{\substack{j,r=1\\ j<r}}^d\vert A_{rj}^{t_i,t_{i+1}}\vert+(1+\vert X_k\vert^2+\vert X_{m_k}\vert^2)\sum_{\substack{j,r=1\\ j<r}}^d\vert A_{rj}^{t_k,t}\vert\Bigg)^{\frac{p}{2}},\\
I_6=&C\mathbb{E}\Bigg(\sum_{i=0}^{k-1}(1+\vert X_{i}\vert^2+\vert X_{m_i}\vert^2)\sum_{\substack{j,r=1\\ j<r}}^d\vert A_{rj}^{t_i,t_{i+1}}\vert^2+(1+\vert X_k\vert^2+\vert X_{m_k}\vert^2)\sum_{\substack{j,r=1\\ j<r}}^d\vert A_{rj}^{t_k,t}\vert^2\Bigg)^{\frac{p}{2}}.
\end{align*}
Firstly, by H\"older's inequality, for any $t\in [0,T]$,
\begin{align}\label{I_1}
I_1=&C\mathbb{E}\left(\int_0^{t}(1+\vert \bar{X}(u)\vert^2+\vert \bar{X}([u])\vert^2){\rm d}u\right)^{\frac{p}{2}}\le C\mathbb{E}\int_0^{t}(1+\vert \bar{X}(u)\vert^p+\vert \bar{X}([u])\vert^p){\rm d}u.
\end{align}
With the help of Lemma \ref{|B(t_1)-B(t_2)|^p} and the following fundamental inequality: if $h_k\ge 0$ and $\sum_{k=1}^n h_k=1$, then $(\sum_{k=1}^{n}h_k|a_k|)^q\le \sum_{k=1}^{n}h_k|a_k|^q$ for $q\ge1$, we can obtain for $t\in [0,T]$ that 
\begin{align}\label{I_2}
I_2=&C\mathbb{E}\Bigg(\sum_{i=0}^{k-1}\frac{\Delta_i}{t}(1+\vert X_{i}\vert^2+\vert X_{m_i}\vert^2)\sum_{j=1}^d\vert \Delta B^j(t_{i+1})\vert^2\frac{t}{\Delta_i}\notag\\
&\qquad+\frac{t-t_k}{t}(1+\vert X_k\vert^2+\vert X_{m_k}\vert^2)\sum_{j=1}^d\vert B^j(t)-B^j(t_k)\vert^2\frac{t}{t-t_k}\Bigg)^{\frac{p}{2}}\notag\\
\le &C\mathbb{E}\Bigg(\sum_{i=0}^{k-1}\Delta_i^{1-\frac{p}{2}}(1+\vert X_{i}\vert^2+\vert X_{m_i}\vert^2)^{\frac{p}{2}}\sum_{j=1}^d\vert\Delta B^j(t_{i+1})\vert^p\notag\\
&\qquad+(t-t_k)^{1-\frac{p}{2}}(1+\vert X_k\vert^2+\vert X_{m_k}\vert^2)^{\frac{p}{2}}\sum_{j=1}^d\vert B^j(t)-B^j(t_k)\vert^p\Bigg)\notag\\
= &C\sum_{i=0}^{k-1}\mathbb{E}\left(\Delta_i^{1-\frac{p}{2}}(1+\vert X_{i}\vert^p+\vert X_{m_i}\vert^p)\mathbb{E}\left(\sum_{j=1}^d\vert\Delta B^j(t_{i+1})\vert^p\vert\mathcal{F}_{t_i}\right)\right)\notag\\
&+C\mathbb{E}\left((t-t_k)^{1-\frac{p}{2}}(1+\vert X_k\vert^p+\vert X_{m_k}\vert^p)\mathbb{E}\left(\sum_{j=1}^d\vert B^j(t)-B^j(t_k)\vert^p\vert\mathcal{F}_{t_k}\right)\right)\notag\\
\le &C\mathbb{E}\left(\sum_{i=0}^{k-1}\Delta_i(1+\vert X_{i}\vert^p+\vert X_{m_i}\vert^p)+(t-t_k)(1+\vert X_k\vert^p+\vert X_{m_k}\vert^p)\right)\notag\\
=&C\mathbb{E}\int_0^{t}(1+\vert \bar{X}(u)\vert^p+\vert \bar{X}([u])\vert^p){\rm d}u.
\end{align}
In the same way, since $t-t_k\le\Delta_k\le 1, \forall k\in\mathbb{N}$, one has
\begin{align}
I_3\le &C\mathbb{E}\Bigg(\sum_{i=0}^{k-1}\Delta_i^{1-\frac{p}{2}}(1+\vert X_{i}\vert^2+\vert X_{m_i}\vert^2)^{\frac{p}{2}}\sum_{j=1}^d\vert\Delta B^j(t_{i+1})\vert^{2p}\notag\\
&\qquad+(t-t_k)^{1-\frac{p}{2}}(1+\vert X_k\vert^2+\vert X_{m_k}\vert^2)^{\frac{p}{2}}\sum_{j=1}^d\vert B^j(t)-B^j(t_k)\vert^{2p}\Bigg)\notag\\
= &C\sum_{i=0}^{k-1}\mathbb{E}\left(\Delta_i^{1-\frac{p}{2}}(1+\vert X_{i}\vert^p+\vert X_{m_i}\vert^p)\sum_{j=1}^d\mathbb{E}\left(\vert\Delta B^j(t_{i+1})\vert^{2p}\vert\mathcal{F}_{t_i}\right)\right)\notag\\
&+C\mathbb{E}\left((t-t_k)^{1-\frac{p}{2}}(1+\vert X_k\vert^p+\vert X_{m_k}\vert^p)\sum_{j=1}^d\mathbb{E}\left(\vert B^j(t)-B^j(t_k)\vert^{2p}\vert\mathcal{F}_{t_k}\right)\right)\notag\\
\le &C\mathbb{E}\left(\sum_{i=0}^{k-1}\Delta_i^{1+\frac{p}{2}}(1+\vert X_{i}\vert^p+\vert X_{m_i}\vert^p)+(t-t_k)^{1+\frac{p}{2}}(1+\vert X_k\vert^p+\vert X_{m_k}\vert^p)\right)\notag\\
\le&C\mathbb{E}\int_0^{t}(1+\vert \bar{X}(u)\vert^p+\vert \bar{X}([u])\vert^p){\rm d}u.\label{I_3}
\end{align}
For $p\ge 4$, it follows from the Burkholder-Davis-Gundy inequality and H\"older's inequality that
\begin{align}\label{I_4}
I_4\le &C\sum_{j=1}^d\mathbb{E}\Bigg(\sum_{i=0}^{k-1}(1+\vert X_{i}\vert^2+\vert X_{m_i}\vert^2)\Delta B^j(t_{i+1})+(1+\vert X_k\vert^2+\vert X_{m_k}\vert^2)( B^j(t)-B^j(t_k))\Bigg)^{\frac{p}{2}}\notag\\
=&C\sum_{j=1}^d\mathbb{E}\Bigg(\int_0^t (1+\vert \bar{X}(u)\vert^2+\vert \bar{X}([u])\vert^2){\rm d}B^j(u)\Bigg)^{\frac{p}{2}}\notag\\
\le &C\mathbb{E}\int_0^t (1+\vert \bar{X}(u)\vert^p+\vert \bar{X}([u])\vert^p){\rm d}u.
\end{align}
Repeating the procedure \eqref{I_2}, using Lemma \ref{Levy area approximation}, it can be similarly deduce that
\begin{align}
I_5\le C\mathbb{E}\int_0^{t}(1+\vert \bar{X}(u)\vert^p+\vert \bar{X}([u])\vert^p){\rm d}u,\label{I_5}\\
I_6 \le C\mathbb{E}\int_0^{t}(1+\vert \bar{X}(u)\vert^p+\vert \bar{X}([u])\vert^p){\rm d}u,\label{I_6}
\end{align}
Substituting \eqref{I_1}-\eqref{I_6} into \eqref{E|X(t)|^p}, we have
\begin{align*}
\mathbb{E}\vert X(t)\vert^p\le C+ C\mathbb{E}\int_0^{t}(1+\vert \bar{X}(u)\vert^p+\vert\bar{X}([u])\vert^p){\rm d}u,
\end{align*}
which gives
\begin{align*}
\sup_{0\le u\le t}\mathbb{E}\vert X(u)\vert^p\le C(1+t)+C\int_0^{t}\sup_{0\le v\le u}\mathbb{E}\vert X(v)\vert^p{\rm d}u.
\end{align*}
Applying the Gronwall inequality, we can finally derive that
\begin{align*}
\sup_{0\le u\le t}\mathbb{E}\vert X(u)\vert^p\le C(1+t){\rm e}^{Ct}, ~\forall t\in [0,T].
\end{align*}
Hence, 
\begin{align*}
\sup_{0\le t\le T}\mathbb{E}\vert X(t)\vert^p\le C.
\end{align*}
The proof is completed.
\end{proof}
The following Lemma shows that the numerical simulation on the time interval $[0,T]$ can be completed in a finite number of time steps.
\begin{lemma}\label{N_T}
Suppose that Eq.\eqref{SDEPCA} satisfies Assumptions \ref{f condition} and \ref{g condition}, $\Delta_k$ is taken as form \eqref{Delta_k} with $\Delta(\cdot)$ satisfies Assumption \ref{step size} and \ref{step size lower bound}. Let $N_T=\max\{k\in\mathbb{N}:t_k<T\}$ and $t_{N_T+1}=T$. Then for any $p>0$ and $M\in\mathbb{N}_+$,
\begin{align*}
\mathbb{E}(N_T+1)^p\le CM^p.
\end{align*}
\end{lemma}
\begin{proof}\rm
Define $N_1=\{k\in\mathbb{N}:t_k=1\}$, according to the definition of $\Delta_k$ and $t_k$, we know that
\begin{itemize}
\item If $\frac{\Delta(X_{0})}{M}\ge 1$, then $\Delta_{0}=1$, $N_1=1$.
\item Otherwise, $N_1\ge 2$, and  
\begin{align*}
\Delta_k=\min\left\{\frac{\Delta(X_k)}{M},1-t_k\right\}=\frac{\Delta(X_k)}{M}, ~\forall k=0,\dots, N_1-2,
\end{align*}
Using Assumption \ref{step size lower bound}, we can deduce that 
\begin{align*}
N_1-1=&\sum_{k=0}^{N_1-2}1=\sum_{k=0}^{N_1-2}\frac{\Delta_k}{\Delta_k}= \sum_{k=0}^{N_1-2}\frac{M}{\Delta(X_k)}\Delta_k\le  \sum_{k=0}^{N_1-2}M(a\vert X_k\vert^q+b)\Delta_k\\
&=\int_{0}^{t_{N_1-1}}M(a\vert \bar{X}(t)\vert^q+b){\rm d}t\le M\int_{0}^{1}(a\vert \bar{X}(t)\vert^q+b){\rm d}t.
\end{align*}
\end{itemize}
Hence, we always have $N_1\le 1+M\int_{0}^{1}(a\vert \bar{X}(t)\vert^q+b){\rm d}t$.

 Similarly, let $N_2=\{k\in\mathbb{N}:t_k=2\}$.
 \begin{itemize}
\item If $\frac{\Delta(X_{N_1})}{M}\ge 1$, then $\Delta_{N_1}=1$, $N_2=N_1+1$.
\item Otherwise, $N_2\ge N_1+2$, and  
\begin{align*}
\Delta_k=\min\left\{\frac{\Delta(X_k)}{M},2-t_k\right\}=\frac{\Delta(X_k)}{M}, ~\forall k=N_1,\dots, N_2-2,
\end{align*}
Using Assumption \ref{step size lower bound}, we can deduce that 
\begin{align*}
N_2-1=&\sum_{k=0}^{N_2-2}1=\sum_{k=0}^{N_1-1}1+\sum_{k=N_1}^{N_2-2}\frac{\Delta_k}{\Delta_k}=N_1+ \sum_{k=N_1}^{N_2-2}\frac{M}{\Delta(X_k)}\Delta_k\\
&\le 1+M\int_{0}^{1}(a\vert \bar{X}(t)\vert^q+b){\rm d}t+ \sum_{k=N_1}^{N_2-2}M(a\vert X_k\vert^q+b)\Delta_k\\
&=1+\int_{0}^{t_{N_2-1}}M(a\vert \bar{X}(t)\vert^q+b){\rm d}t\le 1+M\int_{0}^{2}(a\vert \bar{X}(t)\vert^q+b){\rm d}t.
\end{align*}
\end{itemize}
Hence, $N_2\le 2+M\int_{0}^{2}(a\vert \bar{X}(t)\vert^q+b){\rm d}t$.

Repeating the process above, it can be known that 
\begin{align*}
N_{[T]+1}\le [T]+1+M\int_{0}^{[T]+1}(a\vert \bar{X}(t)\vert^q+b){\rm d}t\le T+1+M\int_{0}^{T+1}(a\vert \bar{X}(t)\vert^q+b){\rm d}t.
\end{align*}
Since $T<[T]+1$, one has
\begin{align*}
N_T\le N_{[T]+1}\le T+1+M\int_{0}^{T+1}(a\vert \bar{X}(t)\vert^q+b){\rm d}t.
\end{align*}
Consequently, with the help of H\"older's inequality, 
\begin{align*}
\mathbb{E}(N_T+1)^p\le& \mathbb{E}\left(T+2+M\int_{0}^{T+1}(a\vert \bar{X}(t)\vert^q+b){\rm d}t\right)^p\\
\le& 2^{p-1}(T+2)^p+2^{p-1}M^p\mathbb{E}\left(\int_{0}^{T+1}(a\vert \bar{X}(t)\vert^q+b){\rm d}t\right)^p\\\le& 2^{p-1}(T+2)^p+2^{p-1}M^p(T+1)^{p-1}\int_{0}^{T+1}\mathbb{E}(a\vert \bar{X}(t)\vert^q+b)^p{\rm d}t.
\end{align*}
By Theorem \ref{numerical solu. bounded}, for any given $M\in\mathbb{N}_+$, we have
\begin{align*}
\mathbb{E}(N_T+1)^p\le CM^p.
\end{align*}
The proof is completed.
\end{proof}
\section{Strong convergence rate of the adaptive Milstein method}\label{convergence rate}
In order to get the strong convergence rate of the adaptive Milstein scheme, we impose some stronger assumptions on the coefficients.
\begin{assumption}\label{f condition_2}
There exists a constant $L_1>0$ such that 
\begin{align*}
\langle x-\bar{x},f(x, y)-f(\bar{x},\bar{y})\rangle \le& L_1(|x-\bar{x}|^2+|y-\bar{y}|^2)
\end{align*} 
for all $x, \bar{x}, y, \bar{y}\in\mathbb{R}^n$.
\end{assumption}

\begin{assumption}\label{g condition_2}
There exists a constant $K_3>0$ such that 
\begin{align*}
\vert g_j(x, y)-g_j(\bar{x},\bar{y})\vert \le K_3(|x-\bar{x}|+|y-\bar{y}|)
\end{align*} 
for all $x, y, \bar{x}, \bar{y}\in\mathbb{R}^n$.
\end{assumption}
\begin{assumption}\label{assumption_4.1}
There exist positive constants $L_2$ and $\gamma\ge 1$ such that 
\begin{align*}
\vert f(x,y)-f(\bar{x},\bar{y})\vert\le L_2(1+|x|^{\gamma}+|y|^{\gamma}+|\bar{x}|^{\gamma}+|\bar{y}|^{\gamma})(\vert x-\bar{x}\vert+\vert y-\bar{y}\vert)
\end{align*} 
for all $x, y, \bar{x}, \bar{y}\in\mathbb{R}^n$.
\end{assumption}
\begin{remark}\label{remark_2}
According to Assumptions \ref{g condition_2} and \ref{assumption_4.1}, it is easy to get that
\begin{align*}
\vert g_j(x,y)\vert\le \tilde{K}(1+\vert x\vert+\vert y\vert),\quad\vert f(x,y)\vert\le \tilde{L}(1+\vert x\vert^{\gamma+1}+\vert y\vert^{\gamma+1}),
\end{align*} 
and 
\begin{align*}
\|D_xg_j(x,y)\|\le K_3,\quad\|D_xf(x,y)\|\le 2L_2(1+\vert x\vert^{\gamma}+\vert y\vert^{\gamma}),
\end{align*} 
where $\tilde{K}=K_3+\max_{j=1,\cdots,d}\vert g_j(0,0)\vert,~\tilde{L}=3L_2+\vert f(0,0)\vert$.
\end{remark}
\begin{assumption}\label{assumption_4.2}
There exist positive constants $L_3$ and $K_4$ such that 
\begin{align*}
\|D_{xx} f(x, y)\|_{T_3} \le& L_3(1+|x|^{\gamma-1}+|y|^{\gamma}),\\
\|D_{xx} g_j(x, y)\|_{T_3} \le& K_4
\end{align*} 
for all $x, y\in\mathbb{R}^n$ and $j=1,2,\dots,d$.
\end{assumption}
\begin{lemma}\label{e_Delta(t)}
Let Assumptions \ref{g condition_2} and \ref{assumption_4.1} hold. Then for any $t\in[0,T]$ and $p>0$, 
\begin{align*}
\mathbb{E}\left(\vert X(t)-\bar{X}(t)\vert ^p\vert\mathcal{F}_{\underline{t}}\right)\le C(1+\vert X(\underline{t})\vert^{(\gamma+1) p}+\vert X([\underline{t}])\vert^{(\gamma+1) p})(t-\underline{t})^{\frac{p}{2}},~a.s.
\end{align*}
\end{lemma}
\begin{proof}\rm
For any $t\in[0,T]$, there are always $k\in\mathbb{N}$ such that $t\in [t_k,t_{k+1})$, by \eqref{continuous Milstein} and \eqref{I_rj^t_k,t}, one has
\begin{align*}
&\vert X(t)-\bar{X}(t)\vert ^p\\
& \le C\left\vert  f(X_k,X_{m_k})(t-t_k)\right\vert ^p+C\left\vert \sum_{j=1}^d g_j(X_k,X_{m_k}) (\Delta B^j(t))\right\vert ^p\\
&\quad+C\left\vert\sum_{j=1}^d D_xg_j(X_k,X_{m_k})g_j(X_k,X_{m_k})\left((\Delta B^j(t))^2-(t-t_k)\right)\right\vert ^p\\
&\quad+C\left\vert\sum_{\substack{j,r=1\\ j<r}}^d \left(D_xg_j(X_k,X_{m_k})g_r(X_k,X_{m_k})+D_xg_r(X_k,X_{m_k})g_j(X_k,X_{m_k})\right)\Delta B^j(t)\Delta B^r(t)\right\vert ^p\\
&\quad+C\left\vert\sum_{\substack{j,r=1\\ j<r}}^d \left(D_xg_r(X_k,X_{m_k})g_j(X_k,X_{m_k})-D_xg_j(X_k,X_{m_k})g_r(X_k,X_{m_k})\right)A_{jr}^{t_k,t}\right\vert ^p
\end{align*}
Then
\begin{align*}
&\mathbb{E}\left(\vert X(t)-\bar{X}(t)\vert ^p\vert\mathcal{F}_{t_k}\right)\\
&\le C\left\vert  f(X_k,X_{m_k})\right\vert ^p\left\vert (t-t_k)\right\vert ^p+C\sum_{j=1}^d\vert g_j(X_k,X_{m_k}) \vert^p\mathbb{E}\left(\vert \Delta B^j(t)\vert ^p\vert\mathcal{F}_{t_k}\right)\\
&\quad+C\sum_{j=1}^d\vert D_xg_j(X_k,X_{m_k})g_j(X_k,X_{m_k})\vert ^p\Big(\mathbb{E}\left(\vert \Delta B^j(t)\vert ^{2p}\vert\mathcal{F}_{t_k}\right)+(t-t_k)^p\Big)\\
&\quad+C\sum_{\substack{j,r=1\\ j<r}}^d\left\vert D_xg_j(X_k,X_{m_k})g_r(X_k,X_{m_k})+D_xg_r(X_k,X_{m_k})g_j(X_k,X_{m_k})\right\vert ^p\\
&~~\qquad\mathbb{E}\left(\vert \Delta B^j(t)\vert ^p\vert\mathcal{F}_{t_k}\right)\mathbb{E}\left(\vert \Delta B^r(t)\vert ^p\vert\mathcal{F}_{t_k}\right)\\
&\quad+C\sum_{\substack{j,r=1\\ j<r}}^d\left\vert D_xg_r(X_k,X_{m_k})g_j(X_k,X_{m_k})-D_xg_j(X_k,X_{m_k})g_r(X_k,X_{m_k})\right\vert ^p\mathbb{E}\left(\vert A_{jr}^{t_k,t}\vert ^p\vert\mathcal{F}_{t_k}\right).
\end{align*}
Applying Remark \ref{remark_2}, Lemma \ref{|B(t_1)-B(t_2)|^p} and \ref{Levy area approximation}, as well as Young's inequality, one can arrive at
\begin{align*}
&\mathbb{E}\left(\vert X(t)-\bar{X}(t)\vert ^p\vert\mathcal{F}_{t_k}\right)\\
\le &C(1+\vert X_k\vert^{(\gamma+1) p}+\vert X_{m_k}\vert^{(\gamma+1) p})(t-t_k)^{p}+C(1+\vert X_k\vert^p+\vert X_{m_k}\vert^p)(t-t_k)^{\frac{p}{2}}\\
&+C(1+\vert X_k\vert^p+\vert X_{m_k}\vert^p)(t-t_k)^{p}\\
\le &C(1+\vert X_k\vert^{(\gamma+1) p}+\vert X_{m_k}\vert^{(\gamma+1) p})(t-t_k)^{\frac{p}{2}},~a.s.
\end{align*}
The proof is completed.
\end{proof}

In the following, let $\varphi:\mathbb{R}^n\times\mathbb{R}^n\to \mathbb{R}^n$ be twice differentiable with respect to the first component, then according to the Taylor formula, 
\begin{align*}
\varphi(x,y)-\varphi(\bar{x},y)=D_x\varphi(\bar{x},y)(x-\bar{x})+R(\varphi)(x-\bar{x}), ~\forall x,y,\bar{x}\in\mathbb{R}^n,
\end{align*}
where 
\begin{align*}
R(\varphi)(x-\bar{x})=&\frac{1}{2}(x-\bar{x})^{\rm T}D_{xx}\varphi(\bar{x}+\theta(x-\bar{x}),y)(x-\bar{x}),
\end{align*}
with $\theta\in (0,1)$. Note that $X([t])=\bar{X}([t])$ for all $t\ge 0$, hence
\begin{align}\label{def_Taylor_2}
\varphi(X(t),X([t]))-\varphi(\bar{X}(t),\bar{X}([t]))=D_x\varphi(\bar{X}(t),X([t]))(X(t)-\bar{X}(t))+R(\varphi)(X(t)-\bar{X}(t)).
\end{align}
According to \eqref{continuous Milstein}, one has
\begin{align*}
X(t)-\bar{X}(t)=&\int_{\underline{t}}^t f(\bar{X}(u),\bar{X}([u])) {\rm d}u+\sum_{j=1}^d\int_{\underline{t}}^t g_{j}(\bar{X}(u),\bar{X}([u])) {\rm d}B^j(u)\\
&+\sum_{j,r=1}^d \int_{\underline{t}}^t D_xg_j(\bar{X}(u),\bar{X}([u]))g_r(\bar{X}(u),\bar{X}([u]))\Delta B^r(u){\rm d}B^j(u).
\end{align*}
Define
\begin{align}\label{def_bar_R}
&\bar{R}(\varphi)(X(t)-\bar{X}(t)):=R(\varphi)(X(t)-\bar{X}(t))+D_x\varphi(\bar{X}(t),X([t])) \int_{\underline{t}}^t f(\bar{X}(u),\bar{X}([u])) {\rm d}u\notag\\
&\quad+D_x \varphi(\bar{X}(t),X([t]))\sum_{j,r=1}^d \int_{\underline{t}}^t D_xg_j(\bar{X}(u),\bar{X}([u]))g_r(\bar{X}(u),\bar{X}([u]))\Delta B^r(u){\rm d}B^j(u),
\end{align}
which gives
\begin{align}\label{fai_X-fai_bar_X}
&\varphi(X(t),X([t]))-\varphi(\bar{X}(t),\bar{X}([t]))\notag\\
&=D_x\varphi(\bar{X}(t),X([t]))\sum_{j=1}^d\int_{\underline{t}}^t g_{j}(\bar{X}(u),\bar{X}([u])) {\rm d}B^j(u)+\bar{R}(\varphi)(X(t)-\bar{X}(t)).
\end{align}
\begin{lemma}\label{lemma_R}
Let Assumptions \ref{step size}, \ref{g condition_2}-\ref{assumption_4.2} hold. Then for any $t\in[0,T]$ and $p>0$,
\begin{align*}
&\mathbb{E}\left(\vert R(\varphi)(X(t)-\bar{X}(t))\vert^p\vert\mathcal{F}_{\underline{t}}\right)\vee\mathbb{E}\left(\vert \bar{R}(\varphi)(X(t)-\bar{X}(t))\vert^p\vert\mathcal{F}_{\underline{t}}\right)\\
\le &C\Big(1+\vert X(\underline{t})\vert^{\bar{p}}+\vert X([\underline{t}])\vert^{\bar{p}}\Big)(t-\underline{t})^{p},~a.s.
\end{align*}
for $\varphi=f, g_j$, $j=1,2,\dots,d$, where $\bar{p}=(\lambda+2\gamma+2) p$.
\end{lemma}
\begin{proof}\rm
Take $\varphi=f$, for any $t\in [0,T]$, using the Cauchy-Schwarz inequality and Assumption \ref{assumption_4.2}, one has
\begin{align*}
&\vert R(f)(X(t)-\bar{X}(t))\vert^p\\
=&2^{-p}\left\vert (X(t)-\bar{X}(t))^{\rm T}D_{xx} f(\bar{X}(t)+\theta(X(t)-\bar{X}(t)),X([t]))(X(t)-\bar{X}(t))\right\vert^p\\
\le& 2^{-p}\left\| D_{xx} f(\bar{X}(t)+\theta(X(t)-\bar{X}(t)),X([t]))\right\|_{T_3}^p\vert X(t)-\bar{X}(t)\vert^{2p}\\
\le& C\left(1+\vert X(t) \vert^{(\gamma-1)p}+\vert\bar{X}(t)\vert^{(\gamma-1) p}+\vert X([t]) \vert^{\gamma p}\right)\vert X(t)-\bar{X}(t)\vert^{2p}.
\end{align*}
Since for any $t\in[0,T]$, there always exist $k\in\mathbb{N}$ such that $t\in[t_k,t_{k+1})$, which gives $\underline{t}=t_k$, hence $\bar{X}(t)=X(\underline{t})=X_k$, and $X([t])=X([\underline{t}])=X([t_k])=X_{m_k}$. By H\"older's inequality, Lemma \ref{e_Delta(t)}, Theorem \ref{numerical solu. bounded} and Young's inequality,  one can obtain that
\begin{align}\label{R_j}
&\mathbb{E}\left(\vert R(f)(X(t)-\bar{X}(t))\vert^p\vert\mathcal{F}_{t_k}\right)\notag\\
\le& C\left(\mathbb{E}\left(1+\vert X(t) \vert^{(\gamma-1) p}+\vert\bar{X}(t)\vert^{(\gamma-1) p}+\vert X([t]) \vert^{\gamma p}\right)^2\Big\vert\mathcal{F}_{t_k}\right)^{\frac{1}{2}}\left(\mathbb{E}\left(\vert X(t)-\bar{X}(t)\vert^{4p}\vert\mathcal{F}_{t_k}\right)\right)^{\frac{1}{2}}\notag\\
\le&  C\Big(1+\mathbb{E}\left(\vert X(t)\vert^{2(\gamma-1) p}\vert\mathcal{F}_{t_k}\right)+\vert X_k\vert^{2(\gamma-1) p}+\vert X_{m_k}\vert^{2\gamma p}\Big)^{\frac{1}{2}}(1+\vert X_k\vert^{2(\gamma+1) p}+\vert X_{m_k}\vert^{2(\gamma+1) p})\Delta_k^{p}\notag\\
\le&  C\Big(1+\vert X_k\vert^{2\gamma p}+\vert X_{m_k}\vert^{2\gamma p}\Big)^{\frac{1}{2}} (1+\vert X_k\vert^{2(\gamma+1) p}+\vert X_{m_k}\vert^{2(\gamma+1) p})\Delta_k^{p}\notag\\
\le&  C\Big(1+\vert X_k\vert^{(3\gamma+2) p}+\vert X_{m_k}\vert^{(3\gamma+2) p}\Big)\Delta_k^{p}, ~a.s.
\end{align}
Moreover, according to \eqref{def_bar_R}, we have
\begin{align}\label{bar_R_j}
&\mathbb{E}\left(\vert \bar{R}(f)(X(t)-\bar{X}(t))\vert^p\vert\mathcal{F}_{t_k}\right)\notag\\
\le &C\mathbb{E}\left(\vert R(f)(X(t)-\bar{X}(t))\vert^p\vert\mathcal{F}_{t_k}\right)+C\mathbb{E}\left(\left\vert D_xf(X_k,X_{m_k}) \int_{t_k}^t f(X_k,X_{m_k}) {\rm d}u\right\vert^p\bigg\vert\mathcal{F}_{t_k}\right)\notag\\
&+C\mathbb{E}\left(\left\vert D_x f(X_k,X_{m_k})\sum_{j,r=1}^d \int_{t_k}^t D_xg_j(X_k,X_{m_k})g_r(X_k,X_{m_k})\Delta B^r(u){\rm d}B^j(u)\right\vert^p\bigg\vert\mathcal{F}_{t_k}\right).
\end{align}
Using the Cauchy-Schwarz inequality, Young's inequality, and Remark \ref{remark_2}, one can deduce that
\begin{align}\label{bar_R_2}
&\mathbb{E}\left(\left\vert D_xf(X_k,X_{m_k}) \int_{t_k}^t f(X_k,X_{m_k}) {\rm d}u\right\vert^p\bigg\vert\mathcal{F}_{t_k}\right)\notag\\
\le&\| D_x f(X_k,X_{m_k})\|^{p}\vert f(X_k,X_{m_k})\vert^{p}\Delta_k^{p}\notag\\
\le& C\left(1+\vert X_k\vert^{\gamma p}+\vert X_{m_k}\vert^{\gamma p}\right)\left(1+\vert X_k\vert^{(\gamma+1) p}+\vert X_{m_k}\vert^{(\gamma+1) p}\right)\Delta_k^{p}\notag  \\
\le& C\left(1+\vert X_k\vert^{(2\gamma+1) p}+\vert X_{m_k}\vert^{(2\gamma+1) p}\right)\Delta_k^{p}.
\end{align}
Similarly, according to \eqref{I_rj^t_k,t}, applying Remark \ref{remark_2}, it yields
\begin{align} 
&\mathbb{E}\left(\left\vert D_x f(X_k,X_{m_k})\sum_{j,r=1}^d \int_{t_k}^t D_xg_j(X_k,X_{m_k})g_r(X_k,X_{m_k})\Delta B^r(u){\rm d}B^j(u)\right\vert^p\bigg\vert\mathcal{F}_{t_k}\right)\notag\\
\le &C\mathbb{E}\left(\left\vert D_x f(X_k,X_{m_k})\sum_{j=1}^d  D_xg_{j}(X_k,X_{m_k})g_{j}(X_k,X_{m_k})\left((\Delta  B^j(t))^2-\Delta_k\right)\right\vert^p\bigg\vert\mathcal{F}_{t_k}\right)\notag\\
&+C\mathbb{E}\Bigg(\Bigg\vert D_x f(X_k,X_{m_k})\sum_{j,r=1,j<r}^d  \Big(D_xg_j(X_k,X_{m_k})g_{r}(X_k,X_{m_k})\notag\\
&+D_xg_r(X_k,X_{m_k})g_{j}(X_k,X_{m_k})\Big)\Delta B^r(t)\Delta B^j(t)\Bigg\vert^p\bigg\vert\mathcal{F}_{t_k}\Bigg)\notag\\
&+C\mathbb{E}\Bigg(\Bigg\vert D_x f(X_k,X_{m_k})\sum_{j,r=1,j<r}^d  \Big(D_xg_r(X_k,X_{m_k})g_j(X_k,X_{m_k})\notag\\
&\qquad-D_xg_j(X_k,X_{m_k})g_r(X_k,X_{m_k})\Big)A_{jr}^{t_k,t}\Bigg\vert^p\bigg\vert\mathcal{F}_{t_k}\Bigg)\notag\\
\le &C\left\|D_x f(X_k,X_{m_k})\right\|^{p}\Bigg( \sum_{j=1}^d\|D_xg_{j}(X_k,X_{m_k})\|^p\vert g_{j}(X_k,X_{m_k})\vert^p\left(\mathbb{E}\left(\vert\Delta  B^j(t)\vert^{2p}\big\vert\mathcal{F}_{t_k}\right)+\Delta_k^p\right)\notag\\
&+\sum_{j,r=1,j\neq r}^d\|D_xg_r(X_k,X_{m_k})\|^p\vert g_{j}(X_k,X_{m_k})\vert^p\mathbb{E}\left(\left\vert  \Delta B^r(t)\right\vert^{p}\big\vert\mathcal{F}_{t_k}\right)\mathbb{E}\left(\vert  \Delta B^j(t)\vert^{p}\big\vert\mathcal{F}_{t_k}\right)\notag\\
&+\sum_{j,r=1,j\neq r}^d\|D_xg_r(X_k,X_{m_k})\|^p\vert g_{j}(X_k,X_{m_k})\vert^p\mathbb{E}\left(\vert  A_{jr}^{t_k,t}\vert^{p}\big\vert\mathcal{F}_{t_k}\right)\Bigg)\notag\\
\le&C\left(1+\vert X_k\vert^{(\gamma+1)p}+\vert X_{m_k}\vert^{(\gamma+1)p}\right)\Delta_k^{p},~a.s.\label{eq_1}
\end{align}
Combining \eqref{R_j}, \eqref{bar_R_j}, \eqref{bar_R_2} and \eqref{eq_1}, with the help of Young's inequality, yields
\begin{align*}
\mathbb{E}\left(\vert \bar{R}(f)(X(t)-\bar{X}(t))\vert^p\vert\mathcal{F}_{\underline{t}}\right)\le C\Big(1+\vert X(\underline{t})\vert^{\bar{p}}+\vert X([\underline{t}])\vert^{\bar{p}}\Big)(t-\underline{t})^{p},~a.s.\quad \forall t\in[0,T],
\end{align*}
where $\bar{p}=(3\gamma+2) p$. Repeating the process above, for all $j=1,2,\dots,d$, we can also obtain that
\begin{align*}
&\mathbb{E}\left(\vert R(g_j)(X(t)-\bar{X}(t))\vert^p\vert\mathcal{F}_{\underline{t}}\right)\vee\mathbb{E}\left(\vert \bar{R}(g_j)(X(t)-\bar{X}(t))\vert^p\vert\mathcal{F}_{\underline{t}}\right)\\
\le& C\Big(1+\vert X(\underline{t})\vert^{\bar{p}}+\vert X([\underline{t}])\vert^{\bar{p}}\Big)(t-\underline{t})^{p},~a.s.
\end{align*}
for any $t\in[0,T]$. The proof is completed.
\end{proof}

\begin{theorem}
Suppose that the SDEPCA \eqref{SDEPCA} satisfies Assumptions \ref{f condition_2}-\ref{assumption_4.2}, $\Delta_k$ is taken as form \eqref{Delta_k} with $\Delta(\cdot)$ satisfies Assumption \ref{step size} and \ref{step size lower bound}. Then for any $p\ge 2$ and $M\ge 1$,
\begin{align*}
\sup_{0\le t\le T}\mathbb{E}\vert x(t)-X(t)\vert^p\le C\mathbb{E}(N_T+1)^{-p},
\end{align*}
where $N_T+1$ is the number of time steps required by a path approximation.
\end{theorem}
\begin{proof}\rm
For any $t\in [0,T]$, there always exist $k\in\mathbb{N}$ such that $t\in[t_k,t_{k+1})$, then for $p\ge 2$, according to \eqref{SDEPCA} and \eqref{continuous Milstein}, using It\^o's formula and the Cauchy-Schwarz inequality, we can arrive at
\begin{align*}
&\vert x(t)-X(t)\vert ^p\le \vert x(t_k)-X_k\vert ^p\\
&+\int_{t_k}^t p\vert x(u)-X(u)\vert ^{p-2}\left((x(u)-X(u))^TF(u)+\frac{p-1}{2}\sum_{j=1}^d\vert G_j(u)\vert ^2\right){\rm d}u\\
&+\sum_{j=1}^d \int_{t_k}^t p\vert x(u)-X(u)\vert ^{p-2}(x(u)-X(u))^TG_j(u){\rm d}B^j(u),
\end{align*}
where 
\begin{align*}
F(u)=&f(x(u),x([u]))-f(\bar{X}(u),\bar{X}([u])),\\
G_j(u)=&g_j(x(u),x([u]))-g_j(\bar{X}(u),\bar{X}([u]))-\sum_{r=1}^{d}D_xg_j(\bar{X}(u),\bar{X}([u]))g_r(\bar{X}(u),\bar{X}([u]))\Delta B^r(u).
\end{align*}
Then, 
\begin{align}\label{A}
\mathbb{E}\left(\vert x(t)-X(t)\vert ^p\vert\mathcal{F}_{t_k}\right)\le &\vert x(t_k)-X_k\vert ^p+\sum_{i=1}^4 A_i,
\end{align}
where
\begin{align*}
A_1=&p\mathbb{E}\left(\int_{t_k}^t \vert x(u)-X(u)\vert ^{p-2}\left(x(u)-X(u)\right)^{\rm T}\left( f(x(u),x([u]))-f(X(u),X([u]))\right){\rm d}u\bigg\vert\mathcal{F}_{t_k}\right),\\
A_2=&p\mathbb{E}\left(\int_{t_k}^t \vert x(u)-X(u)\vert ^{p-2}\left(x(u)-X(u)\right)^{\rm T}\left( f(X(u),X([u]))-f(\bar{X}(u),\bar{X}([u]))\right){\rm d}u\bigg\vert\mathcal{F}_{t_k}\right),\\
A_3=&p(p-1)\mathbb{E}\left(\int_{t_k}^t \vert x(u)-X(u)\vert ^{p-2}\sum_{j=1}^d \vert g_j(x(u),x([u]))-g_j(X(u),X([u]))\vert ^2{\rm d}u\bigg\vert\mathcal{F}_{t_k}\right),\\
A_4=&p(p-1)\mathbb{E}\bigg(\int_{t_k}^t \vert x(u)-X(u)\vert ^{p-2}\sum_{j=1}^d \Bigg\vert g_j(X(u),X([u]))-g_j(\bar{X}(u),\bar{X}([u]))\\
&-\sum_{r=1}^{d}D_xg_j(\bar{X}(u),\bar{X}([u]))g_r(\bar{X}(u),\bar{X}([u]))\Delta B^r(u)\Bigg\vert ^2{\rm d}u\bigg\vert\mathcal{F}_{t_k}\bigg).
\end{align*}
Applying Assumption \ref{f condition_2} and Young's inequality, it is easy to get that
\begin{align}\label{A_1}
A_1\le&pL_1\mathbb{E}\left(\int_{t_k}^t \vert x(u)-X(u)\vert ^{p-2}(\vert x(u)-X(u)\vert ^{2}+\vert x([u])-X([u])\vert ^{2}){\rm d}u\bigg\vert\mathcal{F}_{t_k}\right)\notag\\
\le & C\mathbb{E}\left(\int_{t_k}^t  \left(\vert x(u)-X(u)\vert ^{p}+\vert x([u])-X([u])\vert ^{p}\right){\rm d}u\bigg\vert\mathcal{F}_{t_k}\right).
\end{align}
Similarly, by Assumption \ref{g condition_2}, we can also get 
\begin{align}\label{A_3}
A_3\le C\mathbb{E}\left(\int_{t_k}^t  \left(\vert x(u)-X(u)\vert ^{p}+\vert x([u])-X([u])\vert ^{p}\right){\rm d}u\bigg\vert\mathcal{F}_{t_k}\right).
\end{align}
Next, we give an estimation for $A_4$. According to \eqref{fai_X-fai_bar_X}, let $\varphi=g_j$, we have
\begin{align*}
&g_j(X(u),X([u]))-g_j(\bar{X}(u),\bar{X}([u]))\\
&=D_x g_j(\bar{X}(u),\bar{X}([u]))\sum_{r=1}^d\int_{\underline{u}}^u g_r(\bar{X}(v),\bar{X}([v])) {\rm d}B^r(v)+\bar{R}(g_j)(X(u)-\bar{X}(u)).
\end{align*}
Since $\bar{X}(u)=\bar{X}(v)$, $\bar{X}([u])=\bar{X}([v])$ for $t_k\le v\le u\le t_{k+1}$, hence
\begin{align*}
A_4=p(p-1)\mathbb{E}\bigg(\int_{t_k}^t \vert x(u)-X(u)\vert ^{p-2}\sum_{j=1}^d \left\vert \bar{R}(g_j)(X(u)-\bar{X}(u))\right\vert ^2{\rm d}u\bigg\vert\mathcal{F}_{t_k}\bigg),
\end{align*}
and then according to Lemma \ref{lemma_R},
\begin{align}\label{A_4}
A_4\le&(p-1)\sum_{j=1}^d\mathbb{E}\bigg(\int_{t_k}^t \left((p-2)\vert x(u)-X(u)\vert ^{p}+2\left\vert \bar{R}(g_j)(X(u)-\bar{X}(u))\right\vert ^p\right){\rm d}u\bigg\vert\mathcal{F}_{t_k}\bigg)\notag\\
\le &C\mathbb{E}\bigg(\int_{t_k}^t\vert x(u)-X(u)\vert ^{p}{\rm d}u\bigg\vert\mathcal{F}_{t_k}\bigg)+C\sum_{j=1}^d\int_{t_k}^t\mathbb{E}\left(\vert \bar{R}(g_j)(X(u)-\bar{X}(u))\vert ^p\big\vert\mathcal{F}_{t_k}\right){\rm d}u\notag\\
\le &C\mathbb{E}\bigg(\int_{t_k}^t\vert x(u)-X(u)\vert ^{p}{\rm d}u\bigg\vert\mathcal{F}_{t_k}\bigg)+C\Delta_k^{p}\int_{t_k}^t\Big(1+\vert X_k\vert^{\bar{p}}+\vert X_{m_k}\vert^{\bar{p}}\Big){\rm d}u,~a.s.
\end{align}
In the following we give an estimation for $A_2$. According to \eqref{fai_X-fai_bar_X}, 
\begin{align*}
f(X(u),X([u]))-f(\bar{X}(u),\bar{X}([u]))=\phi(\bar{X}(u),\bar{X}([u]))+\bar{R}(f)(X(u)-\bar{X}(u)),
\end{align*}
where
\begin{align*}
\phi(\bar{X}(u),\bar{X}([u]))=D_x f(X_k,X_{m_k})\sum_{j=1}^d\int_{t_k}^u g_j(X_k,X_{m_k}){\rm d} B^j(v).
\end{align*}
Using H\"older's inequality and Lemma \ref{lemma_R}, one has
\begin{align}\label{A_2_1}
A_2\le&J+p\mathbb{E}\left(\int_{t_k}^t \vert x(u)-X(u)\vert ^{p-1}\vert \bar{R}(f)(X(u)-\bar{X}(u))\vert{\rm d}u\bigg\vert\mathcal{F}_{t_k}\right)\notag\\
\le&J+(p-1)\mathbb{E}\left(\int_{t_k}^t \vert x(u)-X(u)\vert ^p{\rm d}u\bigg\vert\mathcal{F}_{t_k}\right)+\int_{t_k}^t \mathbb{E}\left(\vert \bar{R}(f)(X(u)-\bar{X}(u))\vert ^p\vert\mathcal{F}_{t_k}\right){\rm d}u,
\end{align}
where 
\begin{align*}
J=&p\mathbb{E}\left(\int_{t_k}^t \vert x(u)-X(u)\vert ^{p-2}\left(x(u)-X(u)\right)^{\rm T}\phi(\bar{X}(u),\bar{X}([u])){\rm d}u\bigg\vert\mathcal{F}_{t_k}\right)\\
=&p\sum_{j=1}^d\mathbb{E}\left(\int_{t_k}^t\int_{t_k}^u \vert x(u)-X(u)\vert ^{p-2}\left(x(u)-X(u)\right)^{\rm T}D_x f(X_k,X_{m_k}) g_j(X_k,X_{m_k}) {\rm d}B^j(v){\rm d}u\bigg\vert\mathcal{F}_{t_k}\right).
\end{align*}
Using the integration by parts formula, one can arrive at
\begin{align*}
J=p\sum_{j=1}^d\mathbb{E}\left(\int_{t_k}^t\int_{v}^t \vert x(u)-X(u)\vert ^{p-2}\left(x(u)-X(u)\right)^{\rm T}D_x f(X_k,X_{m_k}) g_j(X_k,X_{m_k}) {\rm d}u{\rm d}B^j(v)\bigg\vert\mathcal{F}_{t_k}\right).
\end{align*}
Let
\begin{align*}
\psi(v)=\int_{v}^t \vert x(u)-X(u)\vert ^{p-2}\left(x(u)-X(u)\right)^{\rm T}D_x f(X_k,X_{m_k}) g_j(X_k,X_{m_k}) {\rm d}u,
\end{align*}
applying H\"older's inequality, the Cauchy-Schwarz inequality, Assumption \ref {assumption_4.1} and Remark \ref{remark_2}, we have
\begin{align*}
 \mathbb{E}\vert\psi(v)\vert^2\le& \mathbb{E}\left\{ (t-v)\int_{v}^t \vert x(u)-X(u)\vert^{2(p-1)}\|D_x f(X_k,X_{m_k}) \|^2\|g_j(X_k,X_{m_k})\| ^2{\rm d}u\right\}\\
\le& C\mathbb{E}\left\{ (t_{k+1}-t_k)\int_{v}^t (\vert x(u)\vert^{2(p-1)}+\vert X(u)\vert^{2(p-1)})(1+\vert X_k\vert^{2(\gamma+1)}+\vert X_{m_k}\vert^{2(\gamma+1)}){\rm d}u\right\}\\
\le& C\mathbb{E}\int_{0}^T (\vert x(u)\vert^{2(p-1)}+\vert X(u)\vert^{2(p-1)})(1+\vert X_k\vert^{2(\gamma+1)}+\vert X_{m_k}\vert^{2(\gamma+1)}){\rm d}u,
\end{align*}
then with the help of Lemma \ref{exact solu. bounded} and Theorem \ref{numerical solu. bounded}, one can know that $\mathbb{E}\vert  \psi(v)\vert^2<\infty$. Consequently, it follows
\begin{align*}
J=p\mathbb{E}\left(\int_{t_k}^{t} \psi(v) {\rm d}B(v)\bigg\vert\mathcal{F}_{t_k}\right)=0.
\end{align*}
Substituting this into \eqref{A_2_1}, using Lemma \ref{lemma_R}, one has
\begin{align}\label{A_2_2}
A_2\le &(p-1)\mathbb{E}\left(\int_{t_k}^t \vert x(u)-X(u)\vert ^p{\rm d}u\bigg\vert\mathcal{F}_{t_k}\right)+\int_{t_k}^t \mathbb{E}\left(\vert \bar{R}(f)(X(u)-\bar{X}(u))\vert ^p\vert\mathcal{F}_{t_k}\right){\rm d}u\notag\\
\le &C\mathbb{E}\left(\int_{t_k}^{t} \vert x(u)-X(u)\vert^p{\rm d}u\bigg\vert\mathcal{F}_{t_k}\right)+C\Delta_k^p\int_{t_k}^{t}(1+\vert X_k\vert^{\bar{p}}+\vert X_{m_k}\vert^{\bar{p}}){\rm d}u,~a.s.
\end{align}
Combining \eqref{A}-\eqref{A_3},\eqref{A_4} and \eqref{A_2_2}, we have
\begin{align*}
\mathbb{E} \left(\vert x(t)-X(t)\vert ^{p}\bigg\vert\mathcal{F}_{t_k}\right)&\le \vert x(t_k)-X_k\vert ^{p}+C\mathbb{E}\left(\int_{t_k}^t  \left(\vert x(u)-X(u)\vert ^{p}+\vert x([u])-X([u])\vert ^{p}\right){\rm d}u\bigg\vert\mathcal{F}_{t_k}\right)\\
&\quad+C\Delta_k^{p}\int_{t_k}^t\Big(1+\vert X_k\vert^{\bar{p}}+\vert X_{m_k}\vert^{\bar{p}}\Big){\rm d}u,~a.s.
\end{align*}
Taking expectations on both sides, using the Tower property of conditional expectations, one can get that 
\begin{align}\label{x(t)-X(t)^p}
\mathbb{E}\vert x(t)-X(t)\vert^p\le& \mathbb{E}\vert x(t_k)-X_k\vert^p+C\mathbb{E}\int_{t_k}^{t} (\vert x(u)-X(u)\vert^p+\vert x([u])-X([u])\vert^p){\rm d}u\notag\\
&+C\mathbb{E}\left\{\Delta_k^{p}\int_{t_k}^{t}\Big(1+\vert X_k\vert^{\bar{p}}+\vert X_{m_k}\vert^{\bar{p}}\Big){\rm d}u\right\}.
\end{align}
In particularly, it follows from the continuity that
\begin{align}\label{X_k-X_k^p}
\mathbb{E}\vert x(t_{k+1})-X(t_{k+1})\vert^p\le& \mathbb{E}\vert x(t_k)-X_k\vert^p+C\mathbb{E}\int_{t_k}^{t_{k+1}} (\vert x(u)-X(u)\vert^p+\vert x([u])-X([u])\vert^p){\rm d}u\notag\\
&+C\mathbb{E}\left\{\Delta_k^{p}\int_{t_k}^{t_{k+1}}\Big(1+\vert X_k\vert^{\bar{p}}+\vert X_{m_k}\vert^{\bar{p}}\Big){\rm d}u\right\}.
\end{align}
Consequently, recall that $x(0)=X_0=x_0$ and $\Delta_k\le \frac{T}{M}$ for all $k\ge 0$, it can be obtained from \eqref{x(t)-X(t)^p} and \eqref{X_k-X_k^p} by iteration that
\begin{align*}
\mathbb{E}\vert x(t)-X(t)\vert^p\le& C\mathbb{E}\int_{0}^{t} (\vert x(u)-X(u)\vert^p+\vert x([u])-X([u])\vert^p){\rm d}u\\
&+C\mathbb{E}\left\{\left(\frac{T}{M}\right)^p\int_{0}^{t}(1+\vert X(\underline{t})\vert^{\bar{p}}+\vert X([\underline{t}])\vert^{\bar{p}}){\rm d}u\right\}.
\end{align*}
Then one can easily know form Lemma \ref{numerical solu. bounded} that
\begin{align*}
\mathbb{E}\vert x(t)-X(t)\vert^p\le& C\mathbb{E}\int_{0}^{t} (\vert x(u)-X(u)\vert^p+\vert x([u])-X([u])\vert^p){\rm d}u+CM^{-p},
\end{align*}
Then it follows from the Fubini theorem that
\begin{align*}
\sup_{0\le u\le t}\mathbb{E}\vert x(u)-X(u)\vert^p\le& C\int_{0}^{t} \sup_{0\le v\le u}\mathbb{E}\vert x(v)-X(v)\vert^p{\rm d}u+CM^{-p}.
\end{align*}
According to the Gronwall inequality and Lemma \ref{N_T}, we have
\begin{align*}
\sup_{0\le t\le T}\mathbb{E} \vert x(t)-X(t)\vert ^{p}\le CM^{-p}e^{CT}\le C\mathbb{E}(N_T+1)^{-p}.
\end{align*}
The proof is completed.
\end{proof}

\section{Numerical examples}\label{examples}
This section presents a couple of examples, they are used as illustrations of the performance of the adaptive Milstein  scheme in different situations. Example \ref{example_1} considers an SDEPCA driven by 1-dimensional Brownian motion which is an analogous equation to the stochastic Ginzburg-Landau equation, while the Brownian motions in Examples \ref{example_2} and \ref{example_3} are both 2-dimensional, whereas the diffusion terms in Example \ref{example_2} satisfy the commutativity condition and those in Example \ref{example_3} do not.
\begin{example}\label{example_1}\rm
Consider the following scalar SDEPCA
\begin{align}\label{example_1_equ.}
{\rm d}x(t)=\big(-x^3(t)+x([t])\big){\rm d}t+(x(t)+ x([t])){\rm d}B(t),~t\in[0,2]
\end{align}
with the initial value $x(0)=2$. 

The numerical solution using the backward Milstein method with uniform step-size $\Delta=2^{-13}$ is taken as the ``exact solution”, we test the convergence of the adaptive Euler, adaptive Milstein and tamed Milstein schemes. The step function in the adaptive method is taken as $\Delta(x)=\frac{2}{1+x^2}$, let $X_k$ denote the numerical approximation of $x(t_k)$, where $t_{k+1}=t_0+\sum_{i=0}^{k}\Delta_i$ with $t_0=0$ and 
\begin{align*}
\Delta_i=\min\left\{\frac{2}{M(1+X_i^2)},[t_i]+1-t_i\right\},
\end{align*}
We take $M=2^3, 2^4, 2^5, 2^6$ and $2^7$ and generate 5000 sample paths for each numerical method. The drift term in the tamed Milstein method is taken to be 
\begin{align*}
f_{\Delta}(x,y)=\frac{f(x,y)}{1+\Delta(|x|^2+|y|^2)}
\end{align*}
with $\Delta=2^{-4},2^{-5},2^{-6},2^{-7}$ and $2^{-8}$. Figure \ref{fig_1} shows the the mean square(MS) error at the terminal time plotted against the number of timesteps. As can be seen from Figure \ref{fig_1}, the convergence order of the adaptive Euler method is about 1/2, while the convergence order of the adaptive Milstein method and the tamed Milstein method can reach 1. Moreover, when the number of steps is the same, the adaptive Milstein method performs better than the tamed Milstein method.
\begin{figure}[htb] \label{fig_1}
	\centering
	\includegraphics[width=4in]{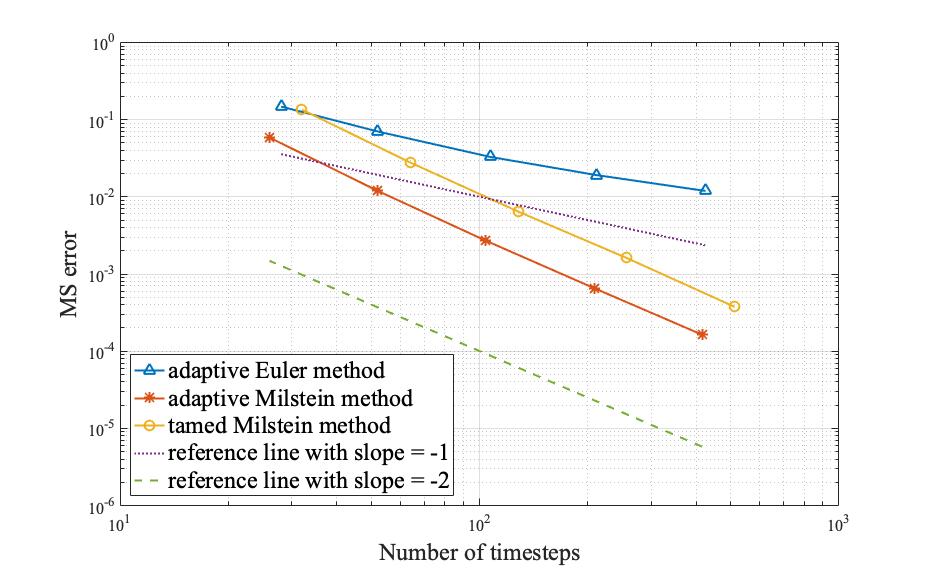}
	\renewcommand{\figurename}{Figure}
	\caption{Loglog plot of mean square errors against timesteps for Example. \ref{example_1}.}
\end{figure}
\end{example}

\begin{example}\label{example_2}\rm
Consider the SDEPCA
\begin{align}\label{example_2_equ.}
{\rm d}x(t)=f(x(t),x([t])){\rm d}t+g(x(t),x([t])){\rm d}B(t),~t\in[0,2]
\end{align}
with the initial value $x(0)=2$, where 
\begin{align*}
f(x,y)=-x^3+x+y,\quad g(x,y)=(5x+y,0.5x+0.1y),
\end{align*}
$B(t)=(B^1(t),B^2(t))^{\rm T}$ is a 2-dimensional Brownian motion. It is easy to see that $f$ and $g$ satisfy all the assumptions in this paper, and
\begin{align*}
D_xg_1(x,y)g_2(x,y)=D_xg_2(x,y)g_1(x,y)=2.5x+0.5y.
\end{align*}
The step function is also taken to be $\Delta(x)=\frac{2}{1+x^2}$, and 
\begin{align*}
\Delta_i=\min\left\{\frac{\Delta(X_i)}{M},[t_i]+1-t_i\right\},
\end{align*}
we take $M=2^6,2^7,2^8$ and $2^9$ in this example. The MS errors of the adaptive Euler, adaptive Milstein and tamed Milstein methods are plotted in Figure \ref{fig_2}.

\begin{figure}[htb] \label{fig_2}
	\centering
	\includegraphics[width=4in]{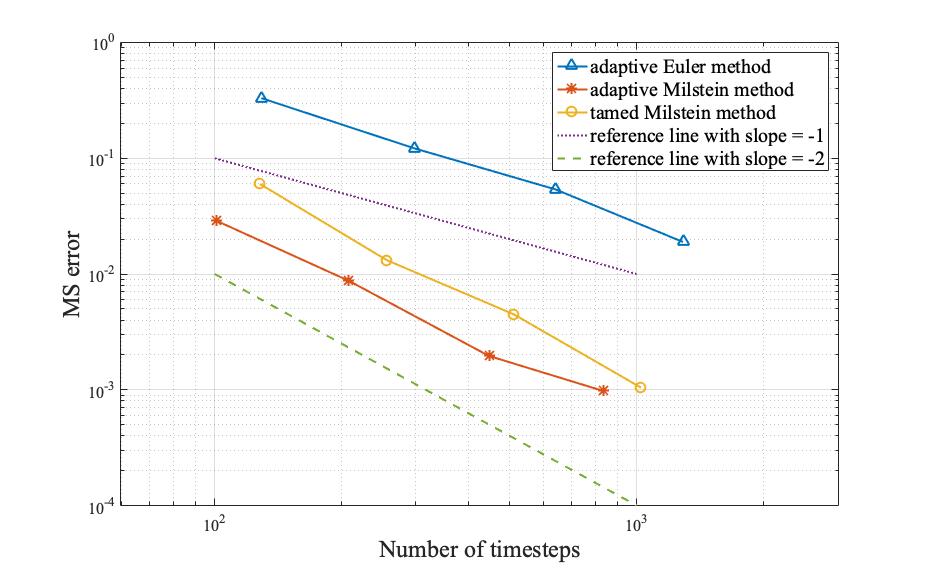}
	\renewcommand{\figurename}{Figure}
	\caption{Loglog plot of errors against step sizes for Example. \ref{example_2}.}
\end{figure}
\end{example}

\begin{example}\label{example_3}\rm
Consider the SDEPCA
\begin{align}\label{example_3_equ.}
{\rm d}x(t)=f(x(t),x([t])){\rm d}t+g(x(t),x([t])){\rm d}B(t),~t\in[0,2]
\end{align}
with the initial value $x(0)=2$, where $B(t)=(B^1(t),B^2(t))^{\rm T}$, 
\begin{align*}
f(x,y)=-x^3+sin(y),\quad g(x,y)=(x+y,-x+y),
\end{align*}
It is easy to see that 
\begin{align*}
D_xg_1(x,y)g_2(x,y)=-x+y,
\end{align*}
while
\begin{align*}
D_xg_2(x,y)g_1(x,y)=-x-y,
\end{align*}
hence the noise is not commutative.

For the SDEPCA with non-commutative noise in this example, we again solve it numerically using the adaptive Euler method, the adaptive Milstein method and the tamed Milstein method. The step function is the same as the previous two examples with $M=2^2,2^3,2^4,2^5$. We use the numerical solution of the backward Euler method with uniform step-size $\Delta=2^{-12}$ as the ``exact solution”. A comparison of the convergence of the three methods is given in Figure \ref{fig_3}.
\begin{figure}[htb] \label{fig_3}
	\centering
	\includegraphics[width=4in]{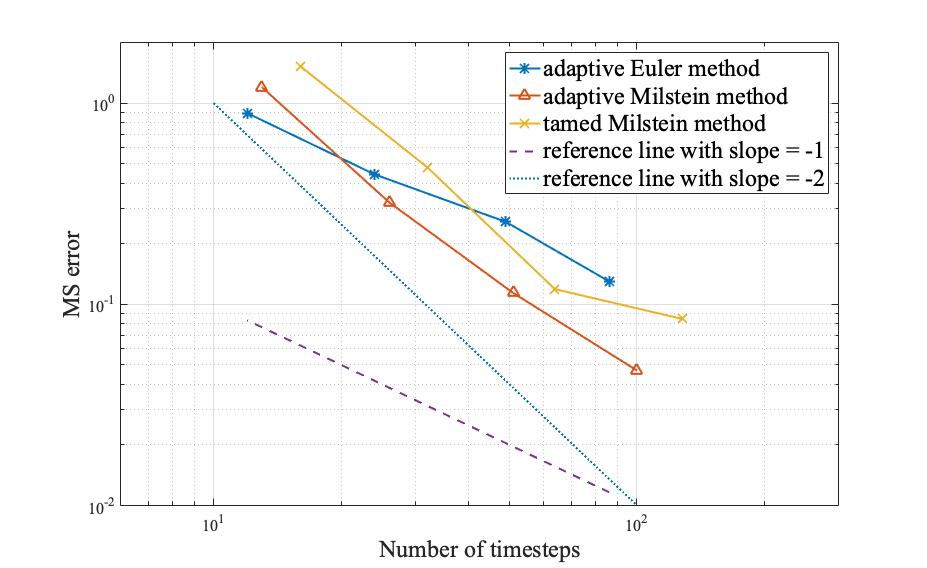}
	\renewcommand{\figurename}{Figure}
	\caption{Loglog plot of mean square errors against timesteps for Example \ref{example_3}.}
\end{figure}
\end{example}
\section*{Declarations} 
\subsection*{Ethical Approval}
Not Applicable.
\subsection*{Availability of supporting data}
Data sharing not applicable to this article as no datasets were generated or analysed during the current study.
\subsection*{Competing interests}
 All authors certify that they have no affiliations with or involvement in any organization or entity with any financial interest or non-financial interest in the subject matter or materials discussed in this manuscript.
\subsection*{Funding}
 This work was Supported by the Postdoctoral Fellowship Program of CPSF under Grant Number GZC20242217, the China Postdoctoral Science Foundation under Grant Number 2024M754160 and the National Natural Science Foundation of China under Grant Numbers 12471372 and 12071101.
 \subsection*{Authors' contributions}
Yuhang Zhang drafted the manuscript and all the authors revised the manuscript together.


\bibliography{sn-bibliography}

\end{document}